\documentclass[UTF-8,final]{siamltex}
\usepackage{amsmath}
\usepackage{epsfig}
\usepackage{graphicx}
\usepackage{amssymb}
\usepackage{CJK}
\usepackage{color}
\usepackage{subfigure} 
\usepackage{float}
\usepackage{hyperref}
\usepackage[noadjust]{cite}

\numberwithin{equation}{section}
\newtheorem{algorithm}{Algorithm}[section]
\newtheorem{remark}{Remark}[section]

\renewcommand\abstractname{Abstract}

\renewcommand\figurename{Figure}
\renewcommand\tablename{Table}

\renewcommand\refname{References}

\allowdisplaybreaks[4]

\normalsize

\def\Ome{{\Omega}}

\def\nab{{\nabla}}

\def\p{{\partial}}

\def\norm#1#2{\Vert\,#1\,\Vert_{#2}}

\def\div{{\mbox{\rm div\,}}}

\def\p{{\partial}}

\def\nab{\nabla}
\def\Ome{\Omega}

\newcommand{\abs}[1]{\left\vert#1\right\vert}

\def\bbf{\mathbf{f}}
\def\bu{\mathbf{u}}

\def\bv{\mathbf{v}}

\def\bg{\mathbf{g}}
\def\bn{\mathbf{n}}
\def\bH{\mathbf{H}}

\def\R{\mathbb{R}}

\begin{document}
	
	
	\title{ Optimal Error Estimates of a new Multiphysic Finite Element Method for Nonlinear Poroelasticity model with Hencky-Mises Stress Tensor\footnote{Last update: \today}}
	
\author{
	Yanan He\thanks{School of Mathematics and Statistics, Henan University, Kaifeng 475004, PR China ({\tt Email:hyn639@163.com}).}
	\and
	Zhihao Ge\thanks{Corresponding author. School of Mathematics and Statistics, Henan University, Kaifeng 475004, PR China ({\tt Email:zhihaoge@henu.edu.cn}).
		The work of this author was supported by the National Natural Science Foundation of China(No. 12371393) and Natural Science Foundation of Henan(No. 242300421047).}
}
	
	\maketitle
	
	
	\setcounter{page}{1}
	
	
	
	\begin{abstract}
		In this paper, we develop a new multiphysics finite element method for a nonlinear poroelastic model with Hencky-Mises stress tensor. By introducing some new notations, we reformulate the original model into a fluid-fluid coupling problem, which is viewed as a generalized nonlinear Stokes sub-problem combined with a reaction-diffusion sub-problem. Then, we establish the existence and uniqueness of the weak solution for the reformulated problem, and propose a stable, fully discrete multiphysics finite element method which employs Lagrangian finite element pairs for spatial discretization and a backward Euler scheme for temporal discretization.
		By ensuring the parameters $\kappa_1$ and $\kappa_3$ remain bounded and non-zero even as $\lambda$ tends to infinity, the proposed method maintains stability for a wide range of Lagrangian element pairs.
		Based on the continuity and monotonicity of the nonlinear term $\mathcal{N}(\varepsilon(\mathbf{u}_h^{n}))$,  we give the stability analysis and derive optimal error estimates for the displacement vector $\mathbf{u}$ and the pressure $p$ in both $L^2$-norm and $H^1$-norm. In particular, the $L^2$-norm error estimate for the displacement $\mathbf{u}$, which was not present in previous literature, is established here through an auxiliary problem and a Poincar$\acute{e}$ inequality.
		Also, we present numerical tests to verify the theoretical analysis, and the results confirm the optimal convergence rates. Finally, we draw conclusions to summarize the work.
	\end{abstract}
	
	
	
	\pagestyle{myheadings}
	\thispagestyle{plain}
	\markboth{YANAN HE, ZHIHAO GE}{ OPTIMAL ERROR ESTIMATES OF A NEW MFEM FOR NONLINEAR POROELASTICITY MODEL}
	

\section{Introduction}\label{sec-1}

Poroelasticity theory describes the coupled mechanical and fluid flow processes in saturated porous media, capturing the time-dependent interaction between solid skeleton deformation and pore fluid flow. This theoretical framework enables the prediction of stress fields, displacement fields, and pore pressure fields in poroelastic media under various loading conditions. Moreover, poroelasticity theory finds extensive applications across multiple disciplines, including geosciences such as earthquakes, landslides, groundwater remediation and petroleum engineering \cite{B1,B3,B4}, environmental engineering such as $\mathrm{CO_{2}}$ sequestration \cite{B7,B8}, and bioengineering for simulating biological soft tissues including arterial walls, myocardium, articular cartilage, and brain swelling \cite{B9,B11,B13}.

The classical poroelastic model, based on linear constitutive relations, has been widely used in the analysis of porous media. However, typical engineering materials such as metals, soils, and concrete exhibit significant nonlinear stress-strain behavior even under small deformations. To more realistically capture such nonlinear responses within a linear strain framework, we introduce the Hencky-Mises nonlinear constitutive relation \cite{B18,Wilber2005,B16} into the poroelastic theory.

This extension retains the mathematical convenience of the classical linear theory while offering other key advantages: firstly, it provides a flexible modeling approach for nonlinear shear behavior through a strain energy function, enabling accurate representation of complex material responses; secondly, its theoretical structure remains consistent with classical elasticity, ensuring theoretical coherence; thirdly, its mathematical form guarantees desirable properties such as monotonicity and coercivity \cite{B18,B16}, laying a theoretical foundation for solution well-posedness and numerical stability.

Building on this, the poroelastic model incorporating  the Hencky-Mises stress tensor is applicable across multiple fields where nonlinear material behavior is crucial.
In metal structures, the core approach integrates small-deformation kinematics with nonlinear constitutive relations, such as those incorporating yielding and hardening \cite{Wriggers2008}.
In soil mechanics, even when overall deformations remain within the small-strain regime, stability issues are strongly influenced by the material’s inherent nonlinear shear and frictional responses, often requiring the use of sophisticated nonlinear material laws for accurate description \cite{Santagata2007, Likitlersuang2013}.
Similarly, in concrete structures, behaviors such as cracking and softening are primarily governed by nonlinear stress-strain relationships at the sectional level and can be effectively described using nonlinear constitutive relations based on the assumption of linear strains \cite{Wriggers2008}.
In biomedical engineering, even under physiological loading that induces only small strains, soft tissues such as articular cartilage and vascular walls still exhibit pronounced nonlinear elasticity \cite{Fung1993, Gasser2006, Safar2018}.
This enhanced modeling capability enables more realistic predictions to be made in applications such as tissue engineering and geological resource extraction.
In this paper, the quasi-static nonlinear poroelastic model is given as follows:
\begin{align} 
	-\div \sigma(\bu) + \alpha \nab p = \bbf
	&\qquad \mbox{in } \Ome_T:=\Ome\times (0,T)\subset \mathbf{\R}^d\times (0,T),\label{1.1}\\
	(c_0p+\alpha \div \bu)_t + \div \bv_f =\phi &\qquad \mbox{in } \Ome_T,	\label{1.2}
\end{align}
where
\begin{align} 	
	\sigma(\mathbf{u})&=\tilde{\lambda}(\mathrm{dev}(\varepsilon(\mathbf{u})))
	\mathrm{tr}(\varepsilon(\mathbf{u}))\mathbf{I}
	+\tilde{\mu}(\mathrm{dev}(\varepsilon(\mathbf{u})))\varepsilon(\mathbf{u}),\label{11-22-1}\\
	\bv_f &= -\frac{K}{\mu_f} \bigl(\nab p -\rho_f \bg \bigr).\label{1.3}
\end{align}

Here, the strain tensor is defined as $\varepsilon(\mathbf{u})=\frac12(\nabla\mathbf{u}+\nabla^T\mathbf{u})$, and the deviatoric operator $\mathrm{dev}:R^{d\times d}\rightarrow R^{+}$ is given by $\mathrm{dev}(\mathbf{\tau}):=\mathrm{tr}(\mathbf{\tau}^2)-\frac{1}{2}\mathrm{tr}^2(\mathbf{\tau})$. 
The stress tensor in \eqref{11-22-1} is governed by a Hencky-Mises constitutive law, which derives from a stored energy density function \cite{B18} $\Psi: \mathbb{R}^{2 \times 2} \to \mathbb{R}$ such that
	\begin{align}\label{2025-6-10-1}
		\sigma_{ij}(\mathbf{u})=\frac{\partial}{\partial \varepsilon_{ij}}\Psi (\varepsilon(\mathbf{u}))\quad\forall i,j\in\{1,2,\cdots,d\}.
	\end{align}
	And we define the stored energy density function as
	\begin{align}\label{11-22-2}
		\Psi(\mathbf{\tau})=\frac{\kappa}{2}\mathrm{tr}^2(\mathrm{\mathbf{\tau}})+\mu
		\Phi(\mathrm{dev}(\mathbf{\tau}))\quad\forall \mathbf{\tau}\in R^{d\times d}.
	\end{align}
	When taking $\kappa=\lambda+\mu$ and $\Phi(\rho)=\rho$ in \eqref{11-22-2}, we can get the linear Cauchy stress tensor 
	\begin{align*}
		\Psi(\mathbf{\tau})=\frac{\lambda}{2}\mathrm{tr}^2(\mathbf{\tau})+\mu\mathrm{tr}(\mathbf{\tau}^2).
	\end{align*}
	Here $\lambda$ and $\mu$ are positive Lam$\acute{e}$  constants and computed from the Young's modulus $E$ and the Poisson ratio $\nu$ by the following formulas
	\begin{align}\label{3.1}
		\lambda=\frac{E\nu}{(1+\nu)(1-2\nu)},\quad\mu=\frac{E}{2(1+\nu)}.
	\end{align}	
	The function $\kappa$ is supposed to be continuous and satisfies 
	\begin{align}\label{11-25-1}
		&0<k_0\leq \kappa(x)\leq k_1 \quad \forall x\in \Omega,
	\end{align}
	where $k_0$ and $k_1$ are positive constants. 
	 The function $\Phi:[0,+\infty]\rightarrow R$ is a function of class $C^{2}$. And for any $\rho\in R^{+}$, there exist positive constants  $C_{1}$, $C_{2}$ and $C_{3}$, such that
	\begin{align}\label{11-25-2}
		 C_{1}\leq \mu\Phi^{'}(\rho)<\kappa,\quad |\rho\Phi^{''}(\rho)|\leq C_{2} \quad and\quad
		\Phi^{'}(\rho)+2\rho\Phi^{''}(\rho)\geq C_{3}.
	\end{align}	
	Using \eqref{2025-6-10-1}-\eqref{11-22-2} and the
	nonlinear Lam$\acute{e}$ functions $\tilde{\lambda},\tilde{\mu}:R^{+}\rightarrow R$ which defined by
	\begin{align*}
		\tilde{\mu}(\rho) :=2\mu\Phi^{'}(\rho),~~\tilde{\lambda}(\rho) :=
		\kappa-\frac{1}{2}\tilde{\mu}(\rho)\quad\forall \rho\in R^{+},
	\end{align*}
	we can get stress tensor \eqref{11-22-1}. 	

To close the system  \eqref{1.1}-\eqref{1.2}, we set the following boundary and initial conditions:
\begin{align} 
	\widehat{\sigma}(\bu,p)\bn=\sigma(\bu)\bn-\alpha p \bn = \bbf_1
	&\qquad \mbox{on } \p\Ome_T:=\p\Ome\times (0,T),\label{1.4}\\
	\bv_f\cdot\bn= -\frac{K}{\mu_f} \bigl(\nab p -\rho_f \bg \bigr)\cdot \bn
	=\phi_1 &\qquad \mbox{on } \p\Ome_T, \label{1.5} \\
	\bu=\bu_0,\qquad p=p_0 &\qquad \mbox{in } \Ome\times\{t=0\}. \label{1.6}
\end{align}

We remark that \eqref{1.1} is the momentum balance equations for the displacement of the medium and \eqref{1.2} is the mass balance equation for the pressure distribution.
Here $\Omega\subset \mathbb{R} ^d\left ( d= 2, 3\right ) $ denotes a bounded convex domain with the boundary $\partial\Omega$. $\mathbf{u}$ denotes the displacement vector of solid and $p$ denotes the pressure of fluid. $\mathbf{f}$ is body force and $\phi$ is a source term. $\sigma(\mathbf{u})$ is called the (effective) stress tensor and $\varepsilon(\mathbf{u})$ is known as the deformed Green strain tensor. $\mathbf{v}_f$ is called Darcy's law. 
The permeability tensor $K=K(x)$ is assumed to be symmetric and uniformly positive definite in the sense that there exists positive constants $K_1$ and $K_2$ such that $K_1|\zeta|^2\leq K(x)\zeta\cdot\zeta\leq K_2|\zeta|^2$ for a.e. $x\in\Omega$ and $\zeta\in\mathbb{R}^d$. 
The coefficient $\alpha>0$ is the Biot-Willis constant used to account for the pressure-deformation coupling and to measure the fluid volume forced out of the solid skeleton by dilation. The constrained specific storage coefficient $c_0\geq0$ is the combined porosity of the medium and compressibility of the fluid and solid.
$\mu_f$ denote the fluid viscosity and  $\mathbf{g}$ denotes the acceleration of gravity. We assume $\rho_{f} \not\equiv 0$. The source terms $\mathbf{f}_1$ and $\phi_1$ are given functions. 


For the classical poroelastic model with linear constitutive relations, 
there are many research results of PDE analysis and numerical analysis. For example, Phillips and Wheeler proposed and analyzed linear poroelastic models for the continuous-in-time case and discrete-in-time case \cite{B14,B15}, respectively. And they found there exist locking phenomenon in computation if $T$ is small in numerical test \cite{Phillips2009}.
Yi developed a nonconforming mixed finite element method to solve Biot's consolidation model \cite{B42}. 
Sun and Rui proposed a coupling of a weak Galerkin method for the displacement of the solid phase with a standard mixed finite method for the pressure and velocity of the fluid phase in poroelastic equation  \cite{B43}.
Feng, Ge and Li
proposed a multiphysics finite element method for a quasi-static poroelastic model by a  multiphysics approach \cite{B27}, and so on.

For elastic problems incorporating the nonlinear Hencky-Mises stress tensor, 
Barrientos, Gatica and Stephan proposed a mixed finite element method using the double saddle point technique to solve nonlinear poroelastic problems, and provided the corresponding weak formulation \cite{B18}. 
Botti, Pietro and Sochala proposed and analyzed a new mixed high-order discretized (linear and nonlinear) elastic models, which are commonly used in solid mechanics in small deformation region \cite{B16}. 
Veiga, Lovadina and Mora proposed a virtual element method for nonlinear poroelastic problems, which mainly focuses on the small deformation region \cite{B17}. Ge and Lou propose a fully discrete multiphysics finite element method for solving a nonlinear poroelasticity model under finite strain \cite{2023gelou}, within whose constitutive framework the Nonlinear Hencky-Mises stress tensor can be regarded as a special case.
While existing research on the poroelastic model with the Hencky-Mises stress tensor remains limited, this work therefore designs a multiphysics finite element method to investigate it.

The innovations of this paper are as follows:
(1) Compared to reference \cite{2023gelou}, we introduce some new notations
$\mathcal{N}(\varepsilon(\mathbf{u}))=\sigma(\mathbf{u})-\frac{1}{\lambda} \mathrm{div}\mathbf{u}\mathbf{I}$, $\xi:=\alpha p -\frac{1}{\lambda} \mathrm{div}\mathbf{u}$
which reformulate the original model into a fluid-fluid coupling problem - this can be viewed as a generalized nonlinear Stokes sub-problem combining with a reaction-diffusion sub-problem. The reformulation model inherently prevents the locking phenomenon.
Moreover, it has been observed that as the value of $\lambda$ increases, the nonlinearity of $\mathcal{N}(\varepsilon(\mathbf{u}))$ becomes stronger.
(2) 
We design a stable, fully discrete multiphysics finite element method and establish optimal error estimates for both the displacement $\mathbf{u}$ and the pressure $p$. By ensuring the parameters $\kappa_1$ and $\kappa_3$ remain bounded and non-zero even as $\lambda \to \infty$, the proposed method maintains stability for a wide range of Lagrangian element pairs (e.g., $P_2$-$P_1$, $P_2$-$P_2$, $P_1$-$P_1$).
(3) 
We derive an $L^2$-norm error estimate for the displacement $\mathbf{u}$, which was unavailable in the existing literature. This estimate is obtained by means of an auxiliary problem and a Poincar$\acute{e}$ inequality.

The remainder of this paper is organized as follows. In Section \ref{2025-9-9-1}, we reformulate the original model based on some new nations and give the existence and uniqueness of weak solution based on the continuity, monotonicity and compulsion of the nonlinear term\ $\mathcal{N} (\varepsilon(\mathbf{u}))$. 
In Section \ref{2025-9-9-2}, we propose a fully discrete multiphysics finite element method to solve the reformulated model with $\mathbf{P}_2-P_1-P_1$ Lagrange elements. We provide a stability analysis and obtain the optimal $L^2$ norm and $H^1$ norm error estimates. Specifically, for the $L^2$ norm error estimate of the displacement $u$, we introduce an auxiliary problem and inequalities.
In Section \ref{2025-9-9-3}, we show some numerical tests to verify the theoretical results and employed the Picard iteration method for the nonlinear terms.
Finally, we draw conclusions to summarize the main results of the paper.

\section{Multiphysics reformulation and PDE analysis}\label{2025-9-9-1}
\subsection{Multiphysics reformulation}
This paper employs standard function space notations, whose definitions can be found in \cite{B21,B28}. We respectively use $( \cdot , \cdot ) $ and $\langle \cdot , \cdot \rangle $ to represent the inner products of $L^2( \Omega) $ and $L^2( \partial\Omega)$. 
Introduce new notations
\begin{align}\label{11-26-2}
	\mathcal{N}(\varepsilon(\mathbf{u}))=\sigma(\mathbf{u})-\frac{1}{\lambda} \mathrm{div}\mathbf{u}\mathbf{I},
	\quad q:= \mathrm{div}\mathbf{u},\quad \eta:=c_{0}p+\alpha q,\quad 
	\xi:=\alpha p -\frac{1}{\lambda} q.
\end{align}
It is easy to check that
\begin{align}\label{3.4}
	p=\kappa_{1} \xi + \kappa_{2} \eta, \qquad q=\kappa_{1} \eta-\kappa_{3} \xi,
\end{align}
where
\begin{align}\label{3.5}
	\kappa_{1}= \frac{\lambda\alpha}{\lambda\alpha^{2}+c_{0}},
	\quad \kappa_{2}=\frac{1}{\lambda\alpha^{2}+c_{0}}, \quad 
	\kappa_{3}=\frac{\lambda c_{0}}{\lambda\alpha^{2}+c_{0}}.
\end{align}
Using \eqref{11-26-2}, the problem \eqref{1.1}-\eqref{1.2} can be rewritten as
\begin{align}
	-\operatorname{div}\mathcal{N}(\varepsilon(\mathbf{u}))+\nabla\xi=\mathbf{f} &\qquad \mathrm{in~} \Omega_T, \label{3.6}\\
	\kappa_3\xi+\mathrm{div}\mathbf{u}=\kappa_1\eta &\qquad \mathrm{in~} \Omega_T, \label{3.7}\\
	\eta_t- \frac{1}{\mu_f}\mathrm{div}[K(\nabla(\kappa_1\xi+\kappa_2\eta)-\rho_f \mathbf{g})]=\phi &\qquad \mathrm{in~} \Omega_T. \label{3.8}
\end{align}
The boundary and initial conditions \eqref{1.4}-\eqref{1.6} can be rewritten as
\begin{align}
	\sigma(\mathbf{u})\mathbf{n}-\alpha (\kappa_1\xi+\kappa_2\eta)\mathbf{n}=\mathbf{f}_1& \qquad \mathrm{on~} \partial\Omega_T,\label{3.9}\\
	-\frac{K}{\mu_f}(\nabla(\kappa_1\xi+\kappa_2\eta)-\rho_f \mathbf{g})\cdot \mathbf{n}=\phi_1& \qquad \mathrm{on~} \partial\Omega_T,\label{3.10}\\
	\mathbf{u}=\mathbf{u_0},~~~p=p_0& \qquad \mathrm{in~}  \Omega\times\{t=0\}.\label{3.11}
\end{align}

In order to ensure the existence and uniqueness of solution of the problem \eqref{3.6}-\eqref{3.8} with pure Neumann boundary conditions, we introduce the infinitesimal rigid motions space
\begin{align}
	\mathbf{RM}:=\{\mathbf{r}\mid\mathbf{r}=\mathbf{a}+\mathbf{b}\times \mathbf{x},~\mathbf{a},\mathbf{b},\mathbf{x}\in\mathbb{R}^d\}.
\end{align}
From \cite{B22,B23}, we know that $\mathbf{RM}$ is the kernel of the strain operator $\varepsilon$, that is, $\mathbf{r}\in \mathbf{RM}$ if and only if $\varepsilon(\mathbf{r})=0$. Hence, we have
\begin{align}
	\varepsilon(\mathbf{r})=0,\quad\div\mathbf{r}=0\qquad\forall\mathbf{r}\in\mathbf{RM}.
\end{align}
Moreover, we define $\mathbf{H}_\bot^1( \Omega)$ as a  subspace of  $\mathbf{H}^1(\Omega)$ and it is orthogonal to $\mathbf{RM}$, that is,
\begin{align}
	\mathbf{H}_\perp^1(\Omega)&:=\{\mathbf{v}\in\mathbf{H}^1(\Omega);\:(\mathbf{v},\mathbf{r})=0\:~~\forall\mathbf{r}\in\mathbf{RM}\}.
\end{align}

Throughout the paper, we assume $C$ is used to denote a generic positive (pure) constant which may be different in different places.
Several fundamental inequalities are presented below.

From \cite{B24}, we know the following Poincar$\acute{e}$ inequality: 
Let $U\in \mathbb{R}^{n}$ be a convex domain with diameter $\gamma$, then
\begin{align}\label{12-12-1}
	\|\phi\|_{L^2(U)}\leq \frac{\gamma}{\pi}\|
	\nabla\phi\|_{L^2(U)}
\end{align}
for all $\phi(x)\in H^1(U)$ satisfying $\int_{U}\phi(x)\mathrm{~dx}=0$.

From \cite{B55}, we know the Korn's inequality holds: 
There exists a positive constant C, such that
\begin{align}\label{11-25-4}
	\sum_{i,j=1}^{2}\|\varepsilon_{ij}(\mathbf{u})\|^2_{L^2(\Omega)}\geq C\|\mathbf{u}\|^2_{H^1(\Omega)}&\quad \forall
	\mathbf{u}\in \mathbf{H}^1(\Omega).
\end{align}

Next, we give the variational format of problem \eqref{3.6}-\eqref{3.8}.
\begin{definition}\label{weak2}
	Let $\mathbf{u}_0\in \mathbf{H}^1(\Omega), \mathbf{f} \in \mathbf{L}^2(\Omega), 
	\mathbf{f}_1 \in \mathbf{L}^2(\partial\Omega),  p_0\in L^2(\Omega), \phi\in L^2(\Omega)$, 
	and $\phi_1\in L^2(\partial\Omega)$.  Assume 
	$(\mathbf{f},\mathbf{v})+\langle \mathbf{f}_1, \mathbf{v} \rangle =0$ for any $\mathbf{v}\in \mathbf{RM}$.
	Given $T > 0$, a $3$-tuple $(\mathbf{u},\xi,\eta)$ with
	\begin{align*}
		&\mathbf{u}\in L^\infty\bigl(0,T; \mathbf{H}_\perp^1(\Omega)), 
		\quad
		\xi\in L^\infty \bigl(0,T; L^2(\Omega)\bigr)
		\cap
		L^2 \bigl(0,T; H^1(\Omega)\bigr)
		, \\
		&\eta\in L^\infty\bigl(0,T; L^2(\Omega)\bigr)
		\cap H^1\bigl(0,T; H^{1}(\Omega)'\bigr).
	\end{align*}
	is called a weak solution to the problem \eqref{3.6}-\eqref{3.8}, if there hold for almost every $t \in (0,T]$
	\begin{align}
		\bigl(\mathcal{N}(\varepsilon(\mathbf{u})), \varepsilon(\mathbf{v}) \bigr)-\bigl( \xi, \operatorname{div} \mathbf{v} \bigr)
		= (\mathbf{f}, \mathbf{v})+\langle \mathbf{f}_1,\mathbf{v}\rangle
		&\qquad\forall \mathbf{v}\in \mathbf{H}^1(\Omega), \label{2.24}\\
		\kappa_3 \bigl( \xi, \varphi \bigr) +\bigl(\operatorname{div}\mathbf{u}, \varphi \bigr)
		= \kappa_1\bigl(\eta, \varphi \bigr) &\qquad\forall \varphi \in L^2(\Omega), \label{2.25}  \\
		\bigl(\eta_t, \psi \bigr)
		+\frac{1}{\mu_f} \bigl(K(\nabla (\kappa_1\xi +\kappa_2\eta) -\rho_f\mathbf{g}), \nabla \psi \bigr)=(\phi,\psi)
		+\langle \phi_1,\psi \rangle
		&\qquad\forall \psi \in H^1(\Omega), \label{2.26} 
	\end{align}
\end{definition}
\subsection{PDE analysis}
In this subsection, we will give the existence and uniqueness analysis of the weak solution of the problem \eqref{2.24}-\eqref{2.26}.
In the beginning, we give some preliminary results for Hencky-Mises stress tensor \eqref{11-22-1}. 
Let $\sigma(\mathbf{u})=a_{ij}(x,\varepsilon(\mathbf{u}))$, where $a_{ij}(x,\varepsilon(\mathbf{u})): \Omega\times\mathbb{R}^{2\times2}\rightarrow \mathbb{R}$ is the set of nonlinear mappings defined by
\begin{align}\label{}
	a_{ij}(x,\varepsilon(\mathbf{u}))&=\tilde{\lambda}(\mathrm{dev}(\varepsilon(\mathbf{u})))
	\mathrm{tr}(\varepsilon(\mathbf{u}))\mathbf{I}
	+\tilde{\mu}(\mathrm{dev}(\varepsilon(\mathbf{u})))\varepsilon(\mathbf{u})\nonumber\\
	&=[\kappa(x)-\frac{1}{2}\tilde{\mu}(\mathrm{dev}(\varepsilon(\mathbf{u})))]\delta_{ij}(\sum_{k=1}^{2}\varepsilon_{kk}(\mathbf{u}))+\tilde{\mu}(\mathrm{dev}(\varepsilon(\mathbf{u})))\varepsilon_{ij}(\mathbf{u})
\end{align}
for all $(x,\varepsilon_{ij}(\mathbf{u}))\in \Omega\times \mathbb{R}^{2\times2}$, with 
\begin{align*}
	\mathrm{dev}(\varepsilon(\mathbf{u}))=\sum_{i,j=1}^{2}\{\varepsilon_{ij}(\mathbf{u})-\frac{1}{2}\delta_{ij}(\sum_{k=1}^{2}\varepsilon_{kk}(\mathbf{u})) \}^2.
\end{align*}

Then, we present the following Lemma \ref{lemma2.3} -  \ref{11-28-2}, which are adopted from reference  \cite{B40} and describe some very important properties of the nonlinear coefficient $a_{ij}$.
\begin{lemma}\label{lemma2.3}
	The function $a_{ij}(\cdot,\varepsilon(\mathbf{u}))$ is measurable in $\Omega$ for all $\varepsilon(\mathbf{u})\in \mathbb{R}^{2\times2}$, and $a_{ij}(x,\cdot)$ is continuous in $\mathbb{R}^{2\times2}$ for almost all $x\in \Omega$. Also there exists $C>0$ such that
	\begin{align*} 
		|a_{ij}(x,\varepsilon(\mathbf{u}))|\leq C|\varepsilon(\mathbf{u})|,
	\end{align*}	
	for all $\varepsilon(\mathbf{u})\in \mathbb{R}^{2\times2}$ and for almost all $x\in \Omega$.
\end{lemma}
\begin{lemma}\label{11-28-1}
	The nonlinear functions $a_{ij}(x,\cdot)$ have first order partial derivatives in $\mathbb{R}^{2\times2}$ for almost all $x\in \Omega$. Also, there exists $C>0$ such that
	\begin{align*}
		\sum_{i,k=1}^{2}\sum_{j,l=1}^{2}\frac{\partial}{\partial\varepsilon_{ij}}a_{ij}(x,\varepsilon(\mathbf{u}))\beta_{kl}\beta_{ij}\geq C \sum_{i,j=1}^{2}\beta_{ij}^2,
	\end{align*}
	for all $\beta:=\beta_{ij}\in \mathbb{R}^{2\times2}$, for all $\varepsilon(\mathbf{u}):=\varepsilon_{ij}(\mathbf{u})\in \mathbb{R}^{2\times2}$ and for almost all $x\in \Omega$.
\end{lemma}
\begin{lemma}\label{11-28-2}
	The nonlinear functions $a_{ij}(x,\cdot)$ have continuous first order partial derivatives in $\mathbb{R}^{2\times2}$ for almost all $x\in\Omega$. Also the functions $\frac{\partial}{\partial\varepsilon_{kl}}a_{ij}(\cdot,\varepsilon(\mathbf{u}))$ are measurable in $\Omega$  for all $\varepsilon(\mathbf{u})\in \mathbb{R}^{2\times2}$, and $\frac{\partial}{\partial\varepsilon_{kl}}a_{ij}(x,\cdot)$ are continuous in $\mathbb{R}^{2\times2}$ for almost all $x\in \Omega$. Also, there exists $C>0$ such that for each $i,j,k,l\in \{1,2\}$
	\begin{align*}
		|\frac{\partial}{\partial\varepsilon_{kl}}a_{ij}(x,\varepsilon(\mathbf{u}))|\leq C\qquad \forall\varepsilon(\mathbf{u})\in \mathbb{R}^{2\times2}, \quad \forall x\in \Omega.
	\end{align*}
\end{lemma}

Using Lemma \ref{11-28-1} and Lemma \ref{11-28-2}, we can obtain the coercivity, continuity and monotonicity of the Hencky-Mises stress tensor \eqref{11-22-1}.
\begin{lemma}
	There exists positive constant $C_1,C_2,C_3$ such that
	\begin{align}
		(\sigma(\mathbf{u}), \varepsilon(\mathbf{u}))&\geq C_1 \|\varepsilon(\mathbf{u})\|^2_{L^2(\Omega)},\label{11-25-6}\\
		\|\sigma(\mathbf{u})-\sigma(\mathbf{v})\|_{L^2(\Omega)}&\leq C_2 \|\varepsilon(\mathbf{u})-\varepsilon(\mathbf{v})\|_{L^2(\Omega)},\label{11-25-7}\\
		(\sigma(\mathbf{u})-\sigma(\mathbf{v}), \varepsilon(\mathbf{u})-\varepsilon(\mathbf{v}))&\geq C_3 \|\varepsilon(\mathbf{u})-\varepsilon(\mathbf{v})\|^2_{L^2(\Omega)}.\label{11-25-8}
	\end{align}	
\end{lemma}
\begin{proof}
	We note that
	\begin{align}\label{11-25-10}
		\sigma(\mathbf{u})-\sigma(\mathbf{v})
		&=a_{ij}(x,\varepsilon(\mathbf{u}))-a_{ij}(x,\varepsilon(\mathbf{v}))\nonumber\\
		&=\int_{0}^{1}\left\{ \sum_{k,l=1}^{2}\frac{\partial }{\partial\alpha_{kl}}a_{ij}(x,\mathbf{\alpha}(x,t))\varepsilon_{kl}(\mathbf{u}-\mathbf{v})
		\right\}\mathrm{dt},
	\end{align}
	with $\mathbf{\alpha}(x,t):=\varepsilon(\mathbf{v})+t\varepsilon(\mathbf{u}-\mathbf{v})$, for any $t\in[0,1]$.
	Using the linearity of $\varepsilon$, \eqref{11-25-10} and Lemma \ref{11-28-1}, we obtain
	\begin{align}
		\sum_{i,j=1}^{2}\int_{\Omega}\{
		a_{ij}(x,\varepsilon(\mathbf{u}))-a_{ij}(x,\varepsilon(\mathbf{v}))\}
		\{\varepsilon_{ij}(\mathbf{u})-\varepsilon_{ij}(\mathbf{v})
		\}\mathrm{dx}\geq C_3 \sum_{i,j=1}^{2}\|\varepsilon_{ij}(\mathbf{u}-\mathbf{v})\|^2_{L^2(\Omega)},
	\end{align}
	which complete the proof of (\ref{11-25-8}).
	Taking $\mathbf{v}=\mathbf{0}$ in (\ref{11-25-8}), we obtain \eqref{11-25-6}. Similarly, using \eqref{11-25-10} and Lemma \ref{11-28-2}, we get \eqref{11-25-7}.
\end{proof}

\begin{lemma}\label{11-28-3}
	There exist positive constant $\tilde{C_1}, \tilde{C_2}, \tilde{C_3}$ such that
	\begin{align}
		(\mathcal{N}(\varepsilon(\mathbf{u})),\varepsilon(\mathbf{u})) &\geq \tilde{C_1}\|\varepsilon(\mathbf{u})\|_{L^{2}(\Omega)}^{2}, \label{4.20}\\
		\|\mathcal{N}(\varepsilon(\mathbf{u}))-\mathcal{N}(\varepsilon(\mathbf{v}))\|_{L^{2}(\Omega)} &\leq \tilde{C_2}\|\varepsilon(\mathbf{u})-\varepsilon(\mathbf{v})\|_{L^{2}(\Omega)}, \label{4.21}\\
		(\mathcal{N}(\varepsilon(\mathbf{u}))-\mathcal{N}(\varepsilon(\mathbf{v})),\varepsilon(\mathbf{u})-\varepsilon(\mathbf{v}))
		&\geq \tilde{C_3}\|\varepsilon(\mathbf{u})-\varepsilon(\mathbf{v})\|_{L^{2}(\Omega)}^{2}. \label{4.22}
	\end{align}
\end{lemma}
\begin{proof}
	First, we know a fact that
	\begin{align}\label{11-25-11}
		\|\operatorname{div}\mathbf{u}\|^2_{L^2(\Omega)}
		\leq 2\sum_{i,j=1}^{2}\|\varepsilon_{ij}(\mathbf{u})\|^2_{L^2(\Omega)}.
	\end{align}
	Then, using \eqref{11-25-8} and \eqref{11-25-11}
	\begin{align}
		(\mathcal{N}(\varepsilon(\mathbf{u}))-\mathcal{N}(\varepsilon(\mathbf{v})),\varepsilon(\mathbf{u})-\varepsilon(\mathbf{v}))
		&=(\sigma(\mathbf{u})-\frac{1}{\lambda}\operatorname{div}\mathbf{u}\mathbf{I}-\sigma(\mathbf{v})+\frac{1}{\lambda}\operatorname{div}\mathbf{v}\mathbf{I},\varepsilon(\mathbf{u})-\varepsilon(\mathbf{v}))\nonumber\\	
		&=(\sigma(\mathbf{u})-\sigma(\mathbf{v}),\varepsilon(\mathbf{u})-\varepsilon(\mathbf{v}))
		-\frac{1}{\lambda}\|\operatorname{div}(\mathbf{u-v})\|^2_{L^2(\Omega)}	\nonumber\\
		&\geq (C_3-\frac{2}{\lambda})\|\varepsilon(\mathbf{u})-\varepsilon(\mathbf{v})\|^2_{L^2(\Omega)},
	\end{align}
	Take a suitable $\lambda$ such that $\tilde{C_3}=C_3-\frac{2}{\lambda}>0$, we obtain \eqref{4.22}. Taking $\mathbf{v}=\mathbf{0}$ in \eqref{4.22}, we get \eqref{4.20}.		
	Similarity, using (\ref{11-25-7}) and \eqref{11-25-11}, we get
	\begin{align}
		\|\mathcal{N}(\varepsilon(\mathbf{u}))-\mathcal{N}(\varepsilon(\mathbf{v}))\|_{L^{2}(\Omega)}
		&=\|\sigma(\mathbf{u})-\frac{1}{\lambda}\operatorname{div}\mathbf{u}\mathbf{I}-\sigma(\mathbf{v})+\frac{1}{\lambda}\operatorname{div}\mathbf{v}\mathbf{I}\|_{L^{2}(\Omega)}\nonumber\\
		&\leq \|\sigma(\mathbf{u})-\sigma(\mathbf{v})\|_{L^{2}(\Omega)}+\frac{1}{\lambda}\|\operatorname{div}(\mathbf{u-v})\mathbf{I}\|_{L^{2}(\Omega)}\nonumber\\
		&\leq (C_2+\frac{2}{\lambda})\|\varepsilon(\mathbf{u})-\varepsilon(\mathbf{u})\|_{L^2(\Omega)},
	\end{align}
	Let $\tilde{C_2}=C_2+\frac{2}{\lambda}$, we obtain \eqref{4.21}.
\end{proof}

\begin{remark}
	Using \eqref{11-26-2} and Lemma $\mathrm{\ref{lemma2.3}}$, we get
	\begin{align}\label{11-26-3}
		\|\sigma(\mathbf{u})\|_{L^2(\Omega)}=\|\mathcal{N}(\varepsilon(\mathbf{u}))+\frac{1}{\lambda}\operatorname{div}\mathbf{u}\mathbf{I}\|_{L^2(\Omega)}
		\leq C\|\varepsilon(\mathbf{u})\|_{L^2(\Omega)}.
	\end{align}
	Using \eqref{11-26-3} and \eqref{11-25-11}, we get
	\begin{align}\label{11-26-4}
		\|\mathcal{N}(\varepsilon(\mathbf{u}))\|_{L^2(\Omega)} \leq \|\frac{1}{\lambda}\operatorname{div}\mathbf{u}\mathbf{I}\|_{L^2(\Omega)} + C\|\varepsilon(\mathbf{u})\|_{L^2(\Omega)}
		\leq C_1 \|\varepsilon(\mathbf{u})\|_{L^2(\Omega)}.
	\end{align}
	Using Cauchy-Schwarz inequality, \eqref{11-26-4} and \eqref{4.20}, we get
	\begin{align}
		\tilde{C_1}\|\varepsilon(\mathbf{u})\|_{L^{2}(\Omega)}^{2}
		\leq
		(\mathcal{N}(\varepsilon(\mathbf{u})), \varepsilon(\mathbf{u}))
		\leq \|\mathcal{N}(\varepsilon(\mathbf{u}))\|_{L^2(\Omega)}\|\varepsilon(\mathbf{u})\|_{L^2(\Omega)}
		\leq C_1 \|\varepsilon(\mathbf{u})\|^2_{L^2(\Omega)}.
	\end{align}
	Then, we deduce there exist a positive constant $\tilde{C}$ satisfying $\tilde{C_1}\leq \tilde{C}\leq C_1$ such that
	\begin{align}\label{11-26-5}
		(\mathcal{N}(\varepsilon(\mathbf{u})), \varepsilon(\mathbf{u}))=\tilde{C}(\varepsilon(\mathbf{u}),\varepsilon(\mathbf{u})).
	\end{align}
\end{remark}
Next, we derive the energy estimate and the existence and uniqueness of the weak solution.
\begin{lemma}\label{11-27-7}
	Every weak solution $\mathbf{u}, \xi, \eta$ of problem \eqref{2.24}-\eqref{2.26} satisfies the following energy law:
	\begin{align}\label{2.47}
		J(t) + \frac{1}{\mu_f} \int_0^t \bigl(K(\nabla(\kappa_{1}\xi+\kappa_{2}\eta)-\rho_f\mathbf{g}), \nabla (\kappa_{1}\xi+\kappa_{2}\eta)\bigr) \mathrm{ds}
		-\int_{0}^{t}(\xi_t,\operatorname{div}\mathbf{u})\mathrm{ds}&\nonumber\\
		-\int_0^t \bigl(\phi, \kappa_{1}\xi+\kappa_{2}\eta\bigr) \mathrm{ds}
		-\int_0^t \langle \phi_1, \kappa_{1}\xi+\kappa_{2}\eta \rangle \mathrm{ds} =J(0)&
	\end{align}
	for all $t\in [0, T]$,  where
	\begin{align}\label{2.48}
		J(t):&=  \tilde{C}\norm{\varepsilon(\mathbf{u}(t))}{L^2(\Ome)}^2+\frac{\kappa_2}{2} \norm{\eta(t)}{L^2(\Ome)}^2 +\frac{\kappa_3}{2} \norm{\xi(t)}{L^2(\Ome)}^2-(\mathbf{f}(t),\mathbf{u}(t))-\langle \mathbf{f}_1(t), \mathbf{u}(t) \rangle .
	\end{align}
\end{lemma}
\begin{proof}
	Setting $\mathbf{v}=\mathbf{u}_t$ in \eqref{2.24} and setting $\varphi = \xi$ in \eqref{2.25} after taking the derivative with respect to time $t$, we get
	\begin{align}
		\bigl(\mathcal{N}(\varepsilon(\mathbf{u})), \varepsilon(\mathbf{u}_t) \bigr)-\bigl( \xi, \operatorname{div} \mathbf{u}_t \bigr)
		&= (\mathbf{f}, \mathbf{u}_t)+\langle \mathbf{f}_1,\mathbf{u}_t\rangle
		, \label{11-27-1}\\
		\kappa_3 \bigl( \xi_t, \xi \bigr) +\bigl(\operatorname{div}\mathbf{u}_t, \xi \bigr)
		&= \kappa_1\bigl(\eta_t, \xi \bigr).\label{11-27-2}
	\end{align}
	Setting $\psi=\kappa_1\xi+\kappa_2\eta$ in \eqref{2.26}, we get
	\begin{align}\label{11-27-3}
		\bigl(\eta_t, \kappa_1\xi+\kappa_2\eta \bigr)
		&+\frac{1}{\mu_f} \bigl(K(\nabla (\kappa_1\xi +\kappa_2\eta) -\rho_f\mathbf{g}), \nabla (\kappa_1\xi+\kappa_2\eta) \bigr)\nonumber\\
		&=\bigl(\phi,\kappa_1\xi+\kappa_2\eta\bigr)
		+\langle \phi_1,\kappa_1\xi+\kappa_2\eta \rangle.
	\end{align}
	Setting $\mathbf{v}=\mathbf{u}$ in \eqref{2.24} after taking the derivative with respect to time $t$, we get
	\begin{align}\label{11-27-4}
		\bigl(\mathcal{N}_t(\varepsilon(\mathbf{u})), \varepsilon(\mathbf{u}) \bigr)-\bigl( \xi_t, \operatorname{div} \mathbf{u} \bigr)
		= (\mathbf{f}_t, \mathbf{u})+\langle \mathbf{f}_{1t},\mathbf{u}\rangle.
	\end{align}
	Adding \eqref{11-27-1}-\eqref{11-27-4} and integrating from $0$ to $t$, we obtain
	\begin{align}
	&\int_{0}^{t}((\mathcal{N}_t(\varepsilon(\mathbf{u})),\varepsilon(\mathbf{u}))+(\mathcal{N}(\varepsilon(\mathbf{u})),\varepsilon(\mathbf{u}_t)))\mathrm{ds}
	+\int_{0}^{t}\kappa_3(\xi_t,\xi)\mathrm{ds}
	+\int_{0}^{t}\kappa_2(\eta_t,\eta)\mathrm{ds}\nonumber\\
	&+ \frac{1}{\mu_f} \int_0^t \bigl(K(\nabla(\kappa_{1}\xi+\kappa_{2}\eta)-\rho_f\mathbf{g}), \nabla (\kappa_{1}\xi+\kappa_{2}\eta)\bigr) \mathrm{ds}
	-\int_{0}^{t}(\xi_t,\operatorname{div}\mathbf{u})\mathrm{ds}\nonumber\\
	&=\int_{0}^{t}[(\mathbf{f},\mathbf{u}_t)+\langle \mathbf{f}_1,\mathbf{u}_t\rangle
	+(\mathbf{f}_t,\mathbf{u})+\langle \mathbf{f}_{1t},\mathbf{u}\rangle
	+(\phi,\kappa_1\xi+\kappa_2\eta)
	+\langle \phi_1,\kappa_1\xi+\kappa_2\eta \rangle]\mathrm{ds}.
   \end{align}
    Using \eqref{11-26-5} and the fact that
    \begin{align*}
    	\frac{d}{dt}(\mathcal{N}(\varepsilon(\mathbf{u})),\varepsilon(\mathbf{u}))=&
    	(\mathcal{N}_t(\varepsilon(\mathbf{u})),\varepsilon(\mathbf{u}))+(\mathcal{N}(\varepsilon(\mathbf{u})),\varepsilon(\mathbf{u}_t)),\\
    	\int_{0}^{t}(\xi_t,\xi)\mathrm{ds}=\int_{\Omega}(\int_{0}^{t}\xi_t\xi\mathrm{ds})\mathrm{d}{\Omega}
    	=&\frac{1}{2}\int_{\Omega}(\xi^2(t)-\xi^2(0))\mathrm{d}\Omega
    	=\frac{1}{2}\|\xi(t)\|^2_{L^2(\Omega)}-\frac{1}{2}\|\xi(0)\|^2_{L^2(\Omega)},
    \end{align*}
    we get
    \begin{align}\label{11-27-5}
    	&\tilde{C}\|\varepsilon(\mathbf{u}(t))\|^2_{L^2(\Omega)}
    	+\frac{\kappa_3}{2}\|\xi(t)\|^2_{L^2(\Omega)}
    	+\frac{\kappa_2}{2}\|\eta(t)\|^2_{L^2(\Omega)}
    	-\int_{0}^{t}(\xi_t,\operatorname{div}\mathbf{u})\mathrm{ds}\nonumber\\
    	&+ \frac{1}{\mu_f} \int_0^t \bigl(K(\nabla(\kappa_{1}\xi+\kappa_{2}\eta)-\rho_f\mathbf{g}), \nabla (\kappa_{1}\xi+\kappa_{2}\eta)\bigr) \mathrm{ds}
    	-(\mathbf{f}(t),\mathbf{u}(t))
    	-\langle \mathbf{f}_{1}(t),\mathbf{u}(t)\rangle
    	\nonumber\\
    	&=\tilde{C}\|\varepsilon(\mathbf{u}(0))\|^2_{L^2(\Omega)}+\frac{\kappa_3}{2}\|\xi(0)\|^2_{L^2(\Omega)}+\frac{\kappa_2}{2}\|\eta(0)\|^2_{L^2(\Omega)}
    	-(\mathbf{f}(0),\mathbf{u}(0))
    	-\langle \mathbf{f}_{1}(0),\mathbf{u}(0)\rangle
    	\nonumber\\
    	&+\int_0^t (\phi, \kappa_{1}\xi+\kappa_{2}\eta) \mathrm{ds}
    	+\int_0^t \langle \phi_1, \kappa_{1}\xi+\kappa_{2}\eta \rangle \mathrm{ds}.
    \end{align}
    The proof is complete.
 \end{proof}
 
 From Lemma \ref{11-27-7}, we immediately have the following prior estimate for the weak solution.
 \begin{lemma}\label{11-28-4}
 	Suppose that $\mathbf{u}_0$ and $p_0$ are sufficiently smooth, then there hold the following estimate for the solution to problem \eqref{2.24}-\eqref{2.26}:
 	 \begin{align}\label{11-27-8}
 		\tilde{C}\|\varepsilon(\mathbf{u}(t))\|^2_{L^2(\Omega)}
 		&+\kappa_3\|\xi(t)\|^2_{L^2(\Omega)}
 		+\frac{\kappa_2}{2}\|\eta(t)\|^2_{L^2(\Omega)}\nonumber\\
 		&+ \frac{1}{\mu_f} \int_0^t \bigl(K(\nabla p-\rho_f\mathbf{g}), \nabla p\bigr) \mathrm{ds}\leq C(N^2_0+M^2_0),
 	\end{align}
 	where
 	\begin{align*}
 		N^2_0=&
 		\|\varepsilon(\mathbf{u}(0))\|^2_{L^2(\Omega)}+\|\xi(0)\|^2_{L^2(\Omega)}+\|\eta(0)\|^2_{L^2(\Omega)}\\
 		&+\|\mathbf{f}(0)\|^2_{L^2(\Omega)}
 		+\|\mathbf{u}(0)\|^2_{L^2(\Omega)}
 		+\|\mathbf{f}_1(0)\|^2_{L^2(\partial\Omega)}
 		+\|\mathbf{u}(0)\|^2_{L^2(\partial\Omega)},\\
 		M^2_0=&\|\mathbf{f}\|^2_{L^\infty(0,T,L^2(\Omega))}
 		+\|\mathbf{f}_1\|^2_{L^\infty(0,T,L^2(\partial\Omega))}
 		+\|\phi\|^2_{L^2(0,T,L^2(\Omega))}
 		+\|\phi_1\|^2_{L^2(0,T,L^2(\partial\Omega))}.
 	\end{align*} 
 \end{lemma}
 \begin{proof}
  Using \eqref{3.7} and Young inequality, we know
  \begin{align}
  	\int_{0}^{t}(\xi_t,\operatorname{div}\mathbf{u})\mathrm{ds}
  	&=\int_{0}^{t}(\xi_t,\kappa_1\eta-\kappa_3\xi)\mathrm{ds}
  	=\int_{0}^{t}(\xi_t,\kappa_1\eta)\mathrm{ds}-\frac{\kappa_3}{2}(\|\xi(t)\|_{L^2(\Omega)}-\|\xi(0)\|_{L^2(\Omega)})\nonumber\\
  	&\leq \frac{1}{2\epsilon}\int_{0}^{t}\|\kappa_1\eta\|^2_{L^2(\Omega)}\mathrm{ds}
  	+\frac{\epsilon}{2}\int_{0}^{t}\|\xi_t\|^2_{L^2(\Omega)}\mathrm{ds}
  	-\frac{\kappa_3}{2}(\|\xi(t)\|_{L^2(\Omega)}-\|\xi(0)\|_{L^2(\Omega)}).
  \end{align}
 Then, the \eqref{11-27-5} becomes
  \begin{align}\label{11-27-6}
 	&\tilde{C}\|\varepsilon(\mathbf{u}(t))\|^2_{L^2(\Omega)}
 	+\kappa_3\|\xi(t)\|^2_{L^2(\Omega)}
 	+\frac{\kappa_2}{2}\|\eta(t)\|^2_{L^2(\Omega)}
 	+ \frac{1}{\mu_f} \int_0^t \bigl(K(\nabla p-\rho_f\mathbf{g}), \nabla p\bigr) \mathrm{ds}\nonumber\\
 	&\leq\tilde{C}\|\varepsilon(\mathbf{u}(0))\|^2_{L^2(\Omega)}+\kappa_3\|\xi(0)\|^2_{L^2(\Omega)}+\frac{\kappa_2}{2}\|\eta(0)\|^2_{L^2(\Omega)}
 	\nonumber\\
 	&\quad+\frac{1}{2\epsilon}\int_{0}^{t}\|\kappa_1\eta\|^2_{L^2(\Omega)}\mathrm{ds}
 	+\frac{\epsilon}{2}\int_{0}^{t}\|\xi_t\|^2_{L^2(\Omega)}\mathrm{ds}
 	+\int_0^t (\phi, p) \mathrm{ds}
 	+\int_0^t \langle \phi_1, p \rangle \mathrm{ds}\nonumber\\
 	&\quad+(\mathbf{f}(t),\mathbf{u}(t))
 	+\langle \mathbf{f}_{1}(t),\mathbf{u}(t)\rangle
 	-(\mathbf{f}(0),\mathbf{u}(0))
 	-\langle \mathbf{f}_{1}(0),\mathbf{u}(0)\rangle\nonumber\\
 	&\leq\tilde{C}\|\varepsilon(\mathbf{u}(0))\|^2_{L^2(\Omega)}+\kappa_3\|\xi(0)\|^2_{L^2(\Omega)}+\frac{\kappa_2}{2}\|\eta(0)\|^2_{L^2(\Omega)}
 	\nonumber\\
 	&\quad+\frac{1}{2\epsilon}\int_{0}^{t}\|\kappa_1\eta\|^2_{L^2(\Omega)}\mathrm{ds}
 	+\frac{\epsilon}{2}\int_{0}^{t}\|\xi_t\|^2_{L^2(\Omega)}\mathrm{ds}
 	+\frac{1}{2\epsilon_1}\int_0^t \|\phi\|^2_{L^2(\Omega)} \mathrm{ds}
 	\nonumber\\
 	&\quad
 	+\frac{\epsilon_1}{2}\int_0^t \|p\|^2_{L^2(\Omega)} \mathrm{ds}
 	+\frac{1}{2\epsilon_2}\int_{0}^{t}\|\phi_1\|^2_{L^2(\partial\Omega)}\mathrm{ds}
 	+\frac{\epsilon_2}{2}\int_{0}^{t}\|p\|^2_{L^2(\partial\Omega)}
     \mathrm{ds}\nonumber\\
 	&\quad+\frac{1}{2\epsilon_3}\|\mathbf{f}(t)\|^2_{L^2(\Omega)}
 	+\frac{\epsilon_3}{2}\|\mathbf{u}(t)\|^2_{L^2(\Omega)}
 	+\frac{1}{2\epsilon_4}\|\mathbf{f}_1(t)\|^2_{L^2(\partial\Omega)}
 	+\frac{\epsilon_4}{2}\|\mathbf{u}(t)\|^2_{L^2(\partial\Omega)}\nonumber\\
 	&\quad+\frac{1}{2}\|\mathbf{f}(0)\|^2_{L^2(\Omega)}
 	+\frac{1}{2}\|\mathbf{u}(0)\|^2_{L^2(\Omega)}
	+\frac{1}{2}\|\mathbf{f}_1(0)\|^2_{L^2(\partial\Omega)}
	+\frac{1}{2}\|\mathbf{u}(0)\|^2_{L^2(\partial\Omega)}.
 \end{align}
 Taking $\epsilon,\epsilon_1,\epsilon_2,\epsilon_3,\epsilon_4$ small enough, and using Gronwall inequality in \eqref{11-27-6}, we can get \eqref{11-27-8}. The proof is complete.
\end{proof}

With the help of the Lemma \ref{11-28-4}, we can derive the solvability of the problem \eqref{2.24}-\eqref{2.26}.
\begin{theorem}\label{123-3}
		Let $\mathbf{u}_0\in\mathbf{H}^1(\Omega), \mathbf{f}\in\mathbf{L}^2(\Omega),
		\mathbf{f}_1\in \mathbf{L}^2(\partial\Omega), p_0\in L^2(\Omega), \phi\in L^2(\Omega)$, and $\phi_1\in L^2(\partial\Omega)$. Suppose $c_0>0$ and $(\mathbf{f},\mathbf{v})+\langle \mathbf{f}_1, \mathbf{v} \rangle =0$ for any $\mathbf{v}\in \mathbf{RM}$. Then there exists a unique weak solution to the problem \eqref{2.24}-\eqref{2.26}. 
\end{theorem}
\begin{proof}
		Firstly, using Lemma \ref{11-28-3}, Lemma \ref{11-28-4} and the Schaefer's fixed point theorem  \cite{Evans2016}, we can prove the existence of a weak solution of the problem \eqref{2.24}-\eqref{2.26}. Here we omit the more details.
		
		Secondly, we prove that the problem \eqref{2.24}-\eqref{2.26} has a unique solution.
		We assume that $(\mathbf{u}_1,\xi_1,\eta_1)$ and $(\mathbf{u}_2,\xi_2,\eta_2)$ are the different solutions of \eqref{2.24}-\eqref{2.26} at the moment $t$. 
		Setting $\psi=1$ in \eqref{2.26} and integrating form $0$ to $t$, we get
		\begin{align}
			\eta(t)= \eta(0)+\int_{0}^{t}[(\phi,1)+ \langle \phi_1,1 \rangle]\mathrm{ds}.
		\end{align}
		It is easy to know $\eta_1(t)=\eta_2(t)$.
		Using \eqref{2.24}-\eqref{2.25}, we can get
		\begin{align}
			\bigl(\mathcal{N}(\varepsilon(\mathbf{u}_1))-\mathcal{N}(\varepsilon(\mathbf{u}_2)), \varepsilon(\mathbf{v}) \bigr)-\bigl( \xi_1-\xi_2, \operatorname{div} \mathbf{v} \bigr)
			= 0
			&\qquad\forall \mathbf{v}\in \mathbf{H}^1(\Omega),\label{11-28-5}\\
			\kappa_3 \bigl( \xi_1-\xi_2, \varphi \bigr) +\bigl(\operatorname{div}\mathbf{u}_1-\operatorname{div}\mathbf{u}_2, \varphi \bigr)
			= 0 &\qquad\forall \varphi \in L^2(\Omega).\label{11-28-6}
		\end{align}
		Setting $\mathbf{v}=\mathbf{u}_1-\mathbf{u}_2$ in \eqref{11-28-5}, $\varphi=\xi_1-\xi_2$ in \eqref{11-28-6} and adding the result, we get
		\begin{align}
			\bigl(\mathcal{N}(\varepsilon(\mathbf{u}_1))-\mathcal{N}(\varepsilon(\mathbf{u}_2)), \varepsilon(\mathbf{u}_1)-\varepsilon(\mathbf{u}_2) \bigr)+\kappa_3\|\xi_1-\xi_2\|^2_{L^2(\Omega)}=0.
		\end{align}
		Using \eqref{4.22}, we obtain
		$\xi_1=\xi_2$ and $\mathbf{u}_1=\mathbf{u}_2+\mathbf{c}$, where $\mathbf{c}$ is an arbitrary constant.
		Since $\mathbf{u}_1,\mathbf{u}_2\in \mathbf{H}_\perp^1(\Omega)$, we get $\mathbf{c}=\mathbf{0}$. The proof is complete.		 
\end{proof}

\section{Fully discrete multiphysics finite element method}\label{2025-9-9-2}
\subsection{Multiphysics finite element method}
Let $\mathcal{T}_h$ be a quasi-uniform triangulation or rectangular partition of $\Omega$ with mesh size $h$, and $\bar{\Omega}=\bigcup_{K\in\mathcal{T}_h}\bar{K}$.
The time interval $(0, T)$ is divided into $N$ equal intervals, denoted by $[t_{n-1},t_n], n=1,2,\cdots, N$ and $\Delta t=\frac{T}{N}$, $\mathrm{then~}t_n=n\Delta t$. 
We define the following space of piecewise polynomials
\begin{align*}
	\mathbf{R}_h &=\bigl\{\mathbf{v}_h\in \mathbf{C}^0(\overline{\Omega});\,
	\mathbf{v}_h|_K\in \mathbf{P}_2(K)~~\forall K\in \mathcal{T}_h \bigr\}, \\
	R_h &=\bigl\{\varphi_h\in C^0(\overline{\Omega});\, \varphi_h|_K\in P_1(K)
	~~\forall K\in \mathcal{T}_h \bigr\},
\end{align*}
where $\mathbf{P}_2(K)$ and $P_1(K)$ are quadratic and linear polynomial spaces on finite element $K$, respectively.
We also introduce the function space
\begin{align}
	L_0^2(\Omega):=\{q\in L^2(\Omega); (q,1)=0\},\quad
	\mathbf{X}_h=\mathbf{R}_h\cap \mathbf{H}^1(\Omega)\quad 
	\mathrm{and}\quad
	M_h = R_h\cap L^2(\Omega).
\end{align}
Moreover, we define
\begin{align}
	\mathbf{V}_h=\{\mathbf{v}_h\in\mathbf{X}_h;~(\mathbf{v}_h,\mathbf{r})=0~~\forall\mathbf{r}\in\mathbf{R}\mathbf{M}\},
\end{align}
it is easy to check that $\mathbf{X}_h=\mathbf{V}_h\bigoplus\mathbf{RM}$.

We adopt $(\mathbf{V}_h,M_h,W_h)$ as the mixed finite element spaces for the variables $(\mathbf{u},\xi,\eta)$.
The finite element approximation space $W_h$ for $\eta$ can be chosen independently, any piecewise polynomial space is acceptable provided that
$M_h \subseteq W_h\subseteq L^2(\Omega)$. 
The most convenient choice is $W_h =M_h$, which
will be adopted in the remainder of this paper.


Note that the spaces $\mathbf{V}_h \times M_h$ is usually referred to as the generalized Taylor-Hood element and satisfied the discrete inf-sup condition \cite{B35,B26},
\begin{align}\label{3.13}
	\|q_h\|_{L^2(\Omega)}\leq \frac{1}{\beta}\sup_{\mathbf{0}\neq\mathbf{v}_h\in \mathbf{V}_h}\frac{(q_h, \div \mathbf{v}_h)}{\|\mathbf{v}_h\|_{H^1(\Omega)}}\qquad\forall q_h\in M_{0h}:=M_h\cap L_0^2(\Omega)
\end{align}
for some constant $\beta>0$.

We recall the following inverse inequality for polynomial function \cite{B28}:
\begin{align}\label{122-8}
	\|v_h\|_{W^{m,s}}\leq Ch^{n-m+\frac{d}{s}-\frac{d}{q}}\|v_h\|_{W^{n,q}}
	\quad
	\forall  v_h\in S_h^r, 0\leq n\leq m \leq1~ and~ 1\leq q\leq s\leq\infty.
\end{align}
\begin{algorithm}\em The fully discrete multiphysics finite element method (MFEM): \label{alg210206-1}
\begin{itemize}
\item[(i)]
Compute $\mathbf{u}^0_h\in \mathbf{V}_h$, $\xi_h^0\in M_h$ and $\eta^0_h\in W_h$ by
\begin{align*}
	\mathbf{u}^0_h &=\mathcal{R}_h\mathbf{u}_0, \quad p^0_h =\mathcal{Q}_hp_0, \quad
	q^0_h =\mathcal{Q}_hq_0 \ (q_0 =\operatorname{div} \mathbf{u}_0), \\
	\eta^0_h &=c_0p^0_h+\alpha q^0_h, \quad
	\xi_h^0 =\alpha p_h^0 -\lambda q_h^0,
\end{align*}
where the definition of $\mathcal{R}_h$ and
$\mathcal{Q}_h$ are in \eqref{4.1} and \eqref{4.6}.

\item[(ii)] For $n=0,1,2, \cdots$,  do the following two steps.
\end{itemize}

 Step 1: Solve for $(\mathbf{u}^{n+1}_h,\xi^{n+1}_h,\eta^{n+1}_h)\in \mathbf{V}_h\times M_h\times W_h$ such that
\begin{align}
	(\mathcal{N}(\varepsilon(\mathbf{u}_{h}^{n+1})), \varepsilon(\mathbf{v}_h) )
	-( \xi^{n+1}_h, \operatorname{div} \mathbf{v}_h )
	=({\mathbf{f}},{\mathbf{v}_{h}})
	+\langle{\mathbf{f}_{1}},{\mathbf{v}_h}\rangle,&\quad\forall{\mathbf{v}_{h}}\in{\mathbf{V}_{h}},
	\label{3.14}\\
	\kappa_{3}(\xi_{h}^{n+1},{\varphi_{h}})
	+(\operatorname{div} \mathbf{u}_{h}^{n+1},{\varphi_{h}})=
	k_{1}(\eta_{h}^{n+\theta},{\varphi_{h}}),&\quad\forall \varphi_{h}\in M_{h},
	\label{3.15}\\
	({d_{t}\eta_{h}^{n+1}},\psi_{h})+\frac{1}{\mu_{f}}({K(\nabla(\kappa_{1}\xi_{h}^{n+1} +\kappa_{2}\eta_{h}^{n+1})-{\rho_{f}}\mathbf{g})},{\nabla\psi_{h}})&\nonumber\\
	=(\phi,\psi_{h})+\langle{\phi_{1}},\psi_{h}\rangle,&\quad\forall \psi_{h}\in M_{h}.
	\label{3.16}
\end{align}

 Step 2: Update $p^{n+1}_h$ by
\begin{align}\label{3.17}
	p^{n+1}_h=\kappa_1\xi^{n+1}_h +\kappa_2\eta^{n+\theta}_h,
\end{align}
where $\theta=0$ or $\theta=1$.
\end{algorithm}

\begin{remark}
	When $\theta=1$, Algorithm $\mathrm{\ref{alg210206-1}}$ is coupled, when $\theta=0$, Algorithm $\mathrm{\ref{alg210206-1}}$ is decouple.
	For the nonlinear terms $(\mathcal{N}(\varepsilon(\mathbf{u}_{h}^{n+1})), \varepsilon(\mathbf{v}_h) )$ in \eqref{3.14}-\eqref{3.15}, we can employ the Newton iterative method, Picard iteration method or the implicit-explicit method.
\end{remark}

\subsection{Stability analysis}
The primary goal of this subsection is to derive a discrete energy law which mimics the Lemma \ref{11-27-7} and Lemma \ref{11-28-4}.
\begin{lemma}\label{122-11}
	Let $\{(\mathbf{u}_h^n,\xi_h^n,\eta_h^n) \}_{n\geq 0}$ be defined by the $MFEM$, then there holds the following inequality:
	\begin{align}
		J_h^{l+1}+S_h^{l+1}=J_h^{0} \quad  \text{for}\quad l\geq 0,
	\end{align}
	where 
	\begin{align*}
		J_h^{l+1}&=\tilde{C}\|\varepsilon(\mathbf{u}_h^{l+1})\|^2_{L^2(\Omega)}
		+\frac{\kappa_3}{2}\|\xi_h^{l+1}\|^2_{L^2(\Omega)}
		+\frac{\kappa_2}{2}\|\eta_h^{l+\theta}\|^2_{L^2(\Omega)}
		-({\mathbf{f}}(t_{l+1}),{\mathbf{u}^{l+1}_{h}})
		-\langle{\mathbf{f}_{1}}(t_{l+1}),{\mathbf{u}^{l+1}_h}\rangle\\
		S_h^{l+1}&=\Delta t\sum_{n=0}^{l}[
		\frac{\kappa_3\Delta t}{2}\|d_t\xi_h^{n+1}\|^2_{L^2(\Omega)}
		+\frac{\kappa_2\Delta t}{2}\|d_t\eta_h^{n+\theta}\|^2_{L^2(\Omega)}
		-( d_t\xi^{n+1}_h, \operatorname{div} \mathbf{u}^{n}_h )
		]\\
		&+\Delta t\sum_{n=0}^{l}\frac{1}{\mu_{f}}({K(\nabla p_h^{n+1}-{\rho_{f}}\mathbf{g})},{\nabla p_h^{n+1}})
		-\Delta t\sum_{n=0}^{l}
		(\phi,p_h^{n+1})+\langle{\phi_{1}},p_h^{n+1}\rangle\\
		&-(1-\theta)\Delta t\sum_{n=0}^{l}
		\frac{\kappa_1\Delta t}{\mu_{f}}({K d_t\nabla\xi_{h}^{n+1}},{\nabla p_h^{n+1}})
		.
	\end{align*}
\end{lemma}
\begin{proof}
	When $\theta=0$, based on \eqref{3.15}, we define $\eta^{-1}_h$ by
	\begin{align}
		\kappa_{3}(\xi_{h}^{0},{\varphi_{h}})
		+(\operatorname{div} \mathbf{u}_{h}^{0},{\varphi_{h}})=
		k_{1}(\eta_{h}^{-1},{\varphi_{h}})\quad\forall \varphi_{h}\in M_{h}.
	\end{align}	
	Setting  $\mathbf{v}_h=d_t\mathbf{u}_h^{n+1}$ in \eqref{3.14}, we get
	\begin{align}\label{122-1}
		(\mathcal{N}(\varepsilon(\mathbf{u}_{h}^{n+1})), d_t\varepsilon(\mathbf{u}^{n+1}_h) )
		-( \xi^{n+1}_h, \operatorname{div} d_t\mathbf{u}^{n+1}_h )
		=({\mathbf{f}}(t_{n+1}),{d_t\mathbf{u}^{n+1}_{h}})
		+\langle{\mathbf{f}_{1}}(t_{n+1}),{d_t\mathbf{u}^{n+1}_h}\rangle.
	\end{align}	
	Applying operator $d_t$ and setting  $\mathbf{v}_h=d_t\mathbf{u}_h^{n}$ in \eqref{3.14}, we obtain
	\begin{align}\label{122-2}
		(d_t\mathcal{N}(\varepsilon(\mathbf{u}_{h}^{n+1})), \varepsilon(\mathbf{u}^{n}_h) )
		-( d_t\xi^{n+1}_h, \operatorname{div} \mathbf{u}^{n}_h )
		=(d_t{\mathbf{f}}(t_{n+1}),{\mathbf{u}^{n}_{h}})
		+\langle{d_t\mathbf{f}_{1}}(t_{n+1}),{\mathbf{u}^{n}_h}\rangle.
	\end{align}	
	Applying operator $d_t$ and setting $\varphi_h=\xi_h^{n+1}$ in \eqref{3.15}, we get
	\begin{align}\label{122-3}
		\kappa_{3}(d_t\xi_{h}^{n+1},\xi_h^{n+1})
		+(\operatorname{div}d_t \mathbf{u}_{h}^{n+1},\xi_h^{n+1})=
		k_{1}(d_t\eta_{h}^{n+\theta},\xi_h^{n+1}).
	\end{align}	
	Setting $\psi_h=p_h^{n+1}=\kappa_1\xi_h^{n+1}+\kappa_2\eta_h^{n}$ after lowing the super index from $n+1$ to $n$ on both sides of \eqref{3.16}, we obtain
	\begin{align}\label{122-4}
		({d_{t}\eta_{h}^{n}},\kappa_1\xi_h^{n+1}+\kappa_2\eta_h^{n})+\frac{1}{\mu_{f}}({K(\nabla(\kappa_{1}\xi_{h}^{n} +\kappa_{2}\eta_{h}^{n})-{\rho_{f}}\mathbf{g})},{\nabla p_h^{n+1}})
		=(\phi,p_h^{n+1})+\langle{\phi_{1}},p_h^{n+1}\rangle.
	\end{align}	
	Adding \eqref{122-1}-\eqref{122-4} and using the fact that
	\begin{align}
		d_t(\mathcal{N}(\varepsilon(\mathbf{u}_h^{n+1})),\varepsilon(\mathbf{u}_h^{n+1}))
		=(d_t\mathcal{N}(\varepsilon(\mathbf{u}_h^{n+1})),\varepsilon(\mathbf{u}_h^{n}))
		+(\mathcal{N}(\varepsilon(\mathbf{u}_h^{n+1})),d_t\varepsilon(\mathbf{u}_h^{n+1})),		
	\end{align}
	we get
	\begin{align}
		d_t(\mathcal{N}(\varepsilon(\mathbf{u}_h^{n+1})),\varepsilon(\mathbf{u}_h^{n+1}))
		&+\kappa_{3}(d_t\xi_{h}^{n+1},\xi_h^{n+1})
		+\kappa_2({d_{t}\eta_{h}^{n}},\eta_h^{n})
		-( d_t\xi^{n+1}_h, \operatorname{div} \mathbf{u}^{n}_h )\nonumber\\
		&+\frac{1}{\mu_{f}}({K(\nabla p_h^{n+1}-{\rho_{f}}\mathbf{g})},{\nabla p_h^{n+1}})
		-\frac{\kappa_1\Delta t}{\mu_{f}}({K d_t\nabla\xi_{h}^{n+1}},{\nabla p_h^{n+1}})\nonumber\\
		&=d_t({\mathbf{f}}(t_{n+1}),{\mathbf{u}^{n+1}_{h}})
		+d_t\langle{\mathbf{f}_{1}}(t_{n+1}),{\mathbf{u}^{n+1}_h}\rangle
		+(\phi,p_h^{n+1})+\langle{\phi_{1}},p_h^{n+1}\rangle.
	\end{align}	
	Using \eqref{11-26-5} and the fact that
	\begin{align}
		(d_t\xi_h^{n+1},\xi_h^{n+1})
		=\frac{\Delta t}{2}\|d_t\xi_h^{n+1}\|^2_{L^2(\Omega)}
		+\frac{1}{2}d_t\|\xi_h^{n+1}\|^2_{L^2(\Omega)},
	\end{align}
	we get
	\begin{align}\label{122-5}
		&d_t \tilde{C}\|\varepsilon(\mathbf{u}_h^{n+1})\|^2_{L^2(\Omega)}
		+\frac{\kappa_3\Delta t}{2}\|d_t\xi_h^{n+1}\|^2_{L^2(\Omega)}
		+\frac{\kappa_3}{2}d_t\|\xi_h^{n+1}\|^2_{L^2(\Omega)}
		+\frac{\kappa_2\Delta t}{2}\|d_t\eta_h^{n}\|^2_{L^2(\Omega)}
		+\frac{\kappa_2}{2}d_t\|\eta_h^{n}\|^2_{L^2(\Omega)}\nonumber\\
		&-( d_t\xi^{n+1}_h, \operatorname{div} \mathbf{u}^{n}_h )
		+\frac{1}{\mu_{f}}({K(\nabla p_h^{n+1}-{\rho_{f}}\mathbf{g})},{\nabla p_h^{n+1}})
		-\frac{\kappa_1\Delta t}{\mu_{f}}({K d_t\nabla\xi_{h}^{n+1}},{\nabla p_h^{n+1}})\nonumber\\
		&=d_t({\mathbf{f}}(t_{n+1}),{\mathbf{u}^{n+1}_{h}})
		+d_t\langle{\mathbf{f}_{1}}(t_{n+1}),{\mathbf{u}^{n+1}_h}\rangle
		+(\phi,p_h^{n+1})+\langle{\phi_{1}},p_h^{n+1}\rangle.
	\end{align}	
	Applying the summation operator $\Delta t\sum_{n=0}^{l}$ to both side of \eqref{122-5}, we yield
	\begin{align}\label{122-6}
		&\tilde{C}\|\varepsilon(\mathbf{u}_h^{l+1})\|^2_{L^2(\Omega)}
		+\frac{\kappa_3}{2}\|\xi_h^{l+1}\|^2_{L^2(\Omega)}
		+\frac{\kappa_2}{2}d_t\|\eta_h^{l}\|^2_{L^2(\Omega)}
		+\Delta t\sum_{n=0}^{l}\frac{1}{\mu_{f}}({K(\nabla p_h^{n+1}-{\rho_{f}}\mathbf{g})},{\nabla p_h^{n+1}})
		\nonumber\\
		&+\Delta t\sum_{n=0}^{l}[
		\frac{\kappa_3\Delta t}{2}\|d_t\xi_h^{n+1}\|^2_{L^2(\Omega)}
		+\frac{\kappa_2\Delta t}{2}\|d_t\eta_h^{n}\|^2_{L^2(\Omega)}
		-( d_t\xi^{n+1}_h, \operatorname{div} \mathbf{u}^{n}_h )
		-\frac{\kappa_1\Delta t}{\mu_{f}}({K d_t\nabla\xi_{h}^{n+1}},{\nabla p_h^{n+1}})
		]\nonumber\\
		&=\tilde{C}\|\varepsilon(\mathbf{u}_h^{0})\|^2_{L^2(\Omega)}
		+\frac{\kappa_3}{2}\|\xi_h^{0}\|^2_{L^2(\Omega)}
		+\frac{\kappa_2}{2}d_t\|\eta_h^{-1}\|^2_{L^2(\Omega)}+(\phi,p_h^{n+1})+\langle{\phi_{1}},p_h^{n+1}\rangle\nonumber\\		
		&+({\mathbf{f}}(t_{l+1}),{\mathbf{u}^{l+1}_{h}})
		+({\mathbf{f}}(t_{0}),{\mathbf{u}^{0}_{h}})
		+\langle{\mathbf{f}_{1}}(t_{l+1}),{\mathbf{u}^{l+1}_h}\rangle
		+\langle{\mathbf{f}_{1}}(t_{0}),{\mathbf{u}^{0}_h}\rangle
		.
	\end{align}	
	When $\theta=1$, similarly, we get
		\begin{align}\label{122-7}
		&\tilde{C}\|\varepsilon(\mathbf{u}_h^{l+1})\|^2_{L^2(\Omega)}
		+\frac{\kappa_3}{2}\|\xi_h^{l+1}\|^2_{L^2(\Omega)}
		+\frac{\kappa_2}{2}d_t\|\eta_h^{l+1}\|^2_{L^2(\Omega)}
		+\Delta t\sum_{n=0}^{l}\frac{1}{\mu_{f}}({K(\nabla p_h^{n+1}-{\rho_{f}}\mathbf{g})},{\nabla p_h^{n+1}})
		\nonumber\\
		&+\Delta t\sum_{n=0}^{l}[
		\frac{\kappa_3\Delta t}{2}\|d_t\xi_h^{n+1}\|^2_{L^2(\Omega)}
		+\frac{\kappa_2\Delta t}{2}\|d_t\eta_h^{n+1}\|^2_{L^2(\Omega)}
		-( d_t\xi^{n+1}_h, \operatorname{div} \mathbf{u}^{n}_h )
		]\nonumber\\
		&=\tilde{C}\|\varepsilon(\mathbf{u}_h^{0})\|^2_{L^2(\Omega)}
		+\frac{\kappa_3}{2}\|\xi_h^{0}\|^2_{L^2(\Omega)}
		+\frac{\kappa_2}{2}d_t\|\eta_h^{0}\|^2_{L^2(\Omega)}+(\phi,p_h^{n+1})+\langle{\phi_{1}},p_h^{n+1}\rangle\nonumber\\		
		&+({\mathbf{f}}(t_{l+1}),{\mathbf{u}^{l+1}_{h}})
		+({\mathbf{f}}(t_{0}),{\mathbf{u}^{0}_{h}})
		+\langle{\mathbf{f}_{1}}(t_{l+1}),{\mathbf{u}^{l+1}_h}\rangle
		+\langle{\mathbf{f}_{1}}(t_{0}),{\mathbf{u}^{0}_h}\rangle
		.
	\end{align}	
	The proof is complete.
\end{proof}
\begin{lemma}
		Let $\{(\mathbf{u}_h^n,\xi_h^n,\eta_h^n) \}_{n\geq 0}$ be defined by the $MFEM$, then there holds the following inequality:
		\begin{align}\label{123-2}
			\tilde{C}\|\varepsilon(\mathbf{u}_h^{l+1})\|^2_{L^2(\Omega)}
			&+\frac{\kappa_3}{2}\|\xi_h^{l+1}\|^2_{L^2(\Omega)}
			+\frac{\kappa_2}{2}\|\eta_h^{l+\theta}\|^2_{L^2(\Omega)}\nonumber\\
			&+\Delta t\sum_{n=0}^{l}\frac{K_1}{2\mu_{f}}
			\|\nabla p_h^{n+1}\|^2_{L^2(\Omega)}
			\leq C(\tilde{N}^2_0+\tilde{M}^2_0) \quad\text{for~} l\geq 0, \theta=0,1
		\end{align}
		provided that $\Delta t=O(h^2)$, where
		\begin{align*}
			\tilde{N}^2_0&=
			\tilde{C}\|\varepsilon(\mathbf{u}_h^{0})\|^2_{L^2(\Omega)}
			+\frac{\kappa_3}{2}\|\xi_h^{0}\|^2_{L^2(\Omega)}
			+\frac{\kappa_2}{2}\|\eta_h^{\theta-1}\|^2_{L^2(\Omega)}
			+\|\mathbf{f}(t_{0})\|^2_{L^2(\Omega)}
			+\|\mathbf{f}_{1}(t_{0})\|^2_{L^2(\partial\Omega)},
			\\
			\tilde{M}^2_0&=\|\mathbf{f}\|^2_{L^{\infty}((0,T);L^2(\Omega))}
			+\|\mathbf{f}_{1}\|^2_{L^{\infty}((0,T);L^2(\Omega))}
			+\|\phi\|^2_{L^{2}((0,T);L^2(\Omega))}
			+\|\phi_{1}\|^2_{L^{2}((0,T);L^2(\Omega))}.
		\end{align*}
\end{lemma}
\begin{proof}
	Using the Cauchy-Schwarz inequality, Young inequality and inverse inequality \eqref{122-8}, we get 
\begin{align}\label{123-1}
	\Delta t\sum_{n=0}^{l}(d_t\xi_h^{n+1},\operatorname{div}\mathbf{u}_h^{n})
	&\leq \Delta t\sum_{n=0}^{l}\|d_t\xi_h^{n+1}\|_{L^{2}(\Omega)}
	\|\operatorname{div}\mathbf{u}_h^{n}\|_{L^{2}(\Omega)}\nonumber\\
	&\leq \Delta t\sum_{n=0}^{l}\frac{\epsilon}{2}
	\|d_t\xi_h^{n+1}\|^2_{L^{2}(\Omega)}
	+\Delta t\sum_{n=0}^{l}\frac{1}{2\epsilon}
	\|\varepsilon(\mathbf{u}_h^{n})\|^2_{L^{2}(\Omega)},
\end{align}
and
\begin{align}\label{122-9}
	\dfrac{\kappa_{1}\Delta t}{\mu_{f}}(Kd_{t}\nabla\xi_{h}^{n+1},\nabla p_{h}^{n+1})
	&\leq \dfrac{K_{2}^{2}\kappa_{1}^{2}}{2K_{1}\mu_{f}}(\Delta t)^2\left\|\nabla d_t\xi_{h}^{n+1} \right\|_{L^{2}(\Omega)}^{2}+\dfrac{K_{1}}{2\mu_{f}}\left\|\nabla p_{h}^{n+1} \right\|_{L^{2}(\Omega)}^{2}\nonumber\\
	&\leq\dfrac{cK_{2}^{2}\kappa_{1}^{2}c_{1}^{2}}{2K_{1}\mu_{f}}\frac{(\Delta t)^2}{h^2}\left\|d_t\xi_{h}^{n+1} \right\|_{L^{2}(\Omega)}^{2}+\dfrac{K_{1}}{2\mu_{f}}\left\|\nabla p_{h}^{n+1} \right\|_{L^{2}(\Omega)}^{2}\nonumber\\
	&\leq\dfrac{\kappa_3\Delta t}{4}\left\|d_t\xi_{h}^{n+1}\right\|_{L^{2}(\Omega)}^{2}
	+\dfrac{K_{1}}{2\mu_{f}}\left\|\nabla p_{h}^{n+1} \right\|_{L^{2}(\Omega)}^{2},
\end{align}
where we assume that
$\Delta t\leq
\frac{K_{1}\mu_{f}\kappa_3}{2cK_{2}^{2}\kappa_{1}^{2}c_{1}^{2}}h^2$.
Using \eqref{123-1}, \eqref{122-9}, taking sufficiently small $\epsilon$, combining with Lemma  \ref{122-11} and applying discrete Gronwall inequality, we can get \eqref{123-2}. The proof is complete.
\end{proof}
\begin{theorem}\label{thm3.4}
	The numerical solution $ \left\lbrace\mathbf{u}_{h}^{n+1},\xi_{h}^{n+1},\eta_{h}^{n+1} \right\rbrace_{n\geq0}  $ of the problem \eqref{3.14}-\eqref{3.16} exists uniquely.
\end{theorem}

The proof of Theorem \ref{thm3.4} is similar to Theorem \ref{123-3}, here we omit the more details.
\subsection{Error estimates}
The goal of this subsection is to analyze the optimal error estimates for the MFEM. To do that, 
for any $\varphi\in L^2(\Omega)$, we define its $L^2$-projection $\mathcal{Q}_h: L^2\rightarrow W_h$ as
\begin{align}\label{4.1}
	\bigl( \mathcal{Q}_h\varphi, \psi_h  \bigr)=\bigl( \varphi, \psi_h  \bigr) \qquad \forall\psi_h\in W_h.
\end{align}
The projection operator $\mathcal{Q}_h: L^2\rightarrow W_h$ satisfies
that for any $\varphi\in H^s(\Omega) (s\geq1)$ \cite{B28},
\begin{align}\label{4.2}
	\|\mathcal{Q}_h\varphi-\varphi\|_{L^2(\Omega)}+h\| \nabla(\mathcal{Q}_h\varphi
	-\varphi) \|_{L^2(\Omega)}\leq Ch^\ell\|\varphi\|_{H^\ell(\Omega)}, \quad \ell=\min\{2, s\}.
\end{align}
Next, for any $\varphi\in H^1(\Omega)$, we define its elliptic projection $\mathcal{S}_h\varphi$ by
\begin{align}
	\bigl(K\nabla\mathcal{S}_h\varphi, \nabla\varphi_h\bigr) &=\bigl(K\nabla\varphi, \nabla\varphi_h\bigr)
	 \qquad \forall \varphi_h\in W_h,\label{4.3}\\
	\bigl(\mathcal{S}_h\varphi, 1\bigr) &=\bigl(\varphi, 1\bigr).
\end{align}
The projection operator $\mathcal{S}_h: H^1(\Omega)\rightarrow W_h$
satisfies that for any $\varphi\in H^s(\Omega) (s>1)$ \cite{B28},
\begin{align}\label{4.5}
	\|\mathcal{S}_h\varphi-\varphi\|_{L^2(\Omega)}+h\| \nabla(\mathcal{S}_h\varphi-\varphi) \|_{L^2(\Omega)}
	\leq Ch^\ell\|\varphi\|_{H^\ell(\Omega)}, \quad \ell=\min\{2, s\}.
\end{align}
Finally, for any $\mathbf{v}\in \mathbf{H}^1_\perp(\Omega)$, we define its elliptic projection $\mathcal{R}_h\mathbf{v}$ by
\begin{align}\label{4.6}
	\bigl(\varepsilon(\mathcal{R}_h\mathbf{v}), \varepsilon(\mathbf{w}_h)\bigr)
	=\bigl(\varepsilon(\mathbf{v}), \varepsilon(\mathbf{w}_h)\bigr) \quad \mathbf{w}_h\in \mathbf{V}_h,
\end{align}
The projection $\mathcal{R}_h\mathbf{v}$ satisfies that for any
$\mathbf{v}\in \bH^1_\perp(\Omega)\cap \mathbf{H}^s(\Omega) (s>1)$ \cite{B28}, 
\begin{align}\label{4.7}
	\|\mathcal{R}_h\mathbf{v}-\mathbf{v}\|_{L^2(\Omega)}+h\| \nabla(\mathcal{R}_h\mathbf{v}-\mathbf{v}) \|_{L^2(\Omega)}
	\leq Ch^m\|\mathbf{v}\|_{H^m(\Ome)}, m=\min\{3, s\}.
\end{align} 
To derive the error estimates, we introduce the following notations
\begin{align*}	E_{\mathbf{u}}^{n+1}=\mathbf{u}(t_{n+1})-\mathbf{u}_{h}^{n+1},\quad E_{\xi}^{n+1}=\xi(t_{n+1})-\xi_{h}^{n+1},\\
E_{\eta}^{n+1}=\eta(t_{n+1})-\eta_{h}^{n+1},
\quad
E_{p}^{n+1}=p(t_{n+1})-p_{h}^{n+1}.
\end{align*}
It is easy to check out
\begin{align}
	E_{p}^{n+1}=\kappa_{1}E_{\xi}^{n+1}+\kappa_{2}E_{\eta}^{n+1}+(1-\theta)\kappa_2\Delta t d_t\eta_h^{n+1}.
\end{align}
Also, we denote
\begin{align*}
	&E_\mathbf{u}^n =\mathbf{u}(t_n)-\mathcal{R}_h(\mathbf{u}(t_n))+\mathcal{R}_h(\mathbf{u}(t_n))-\mathbf{u}_h^n:=\Lambda_\mathbf{u}^n+\Theta_\mathbf{u}^n, \\
	&E_{\xi}^n =\xi(t_n)-\mathcal{Q}_h(\xi(t_n))+\mathcal{Q}_h(\xi(t_n))-\xi_h^n:=\Lambda_\xi^n+\Theta_\xi^n, \\
	&E_{\xi}^n =\xi(t_n)-\mathcal{S}_h(\xi(t_n))+\mathcal{S}_h(\xi(t_n))-\xi_h^n:=\hat{\Lambda}_\xi^n+\hat{\Theta}_\xi^n, \\
	&E_{\eta}^{n} =\eta(t_n)-\mathcal{Q}_h(\eta(t_n))+\mathcal{Q}_h(\eta(t_n))-\eta_h^n:=\Lambda_\eta^n+\Theta_\eta^n, \\
	&E_{\eta}^{n} =\eta(t_n)-\mathcal{S}_h(\eta(t_n))+\mathcal{S}_h(\eta(t_n))-\eta_h^n:=\hat{\Lambda}_\eta^n+\hat{\Theta}_\eta^n, \\
	&E_{p}^{n} =p(t_n)-\mathcal{Q}_h(p(t_n))+\mathcal{Q}_h(p(t_n))-p_h^n:=\Lambda_p^n+\Theta_p^n,\\
	&E_{p}^{n} =p(t_n)-\mathcal{S}_h(p(t_n))+\mathcal{S}_h(p(t_n))-p_h^n:=\hat{\Lambda}_p^n+\hat{\Theta}_p^n. 
\end{align*}
\begin{lemma}\label{lemma4.1}
	Let $ \left\lbrace (\mathbf{u}_{h}^{n}, \xi_{h}^{n}, \eta_{h}^{n})\right\rbrace_{n\geq0}  $ be generated by the MFEM. Then there holds
	\begin{align}\label{4.9}
	\max_{0\leq n\leq l}\left[ \tilde{C}_3\|\varepsilon(\Theta_{\mathbf{u}}^{n+1})\|^2_{L^2(\Omega)}\right.
	&+\left.\frac{\kappa_3}{2}\|\Theta_{\xi}^{n+1}\|^2_{L^2(\Omega)}
	+\frac{\kappa_2}{2}\|\Theta_{\eta}^{n+\theta}\|^2_{L^2(\Omega)}\right] \nonumber\\
	&+\Delta t\sum_{n=0}^{l}\frac{K_1}{\mu_{f}}\|\nabla\hat{Z}_p^{n+\theta}\|^2_{L^2(\Omega)}
	\leq C_1(T)(\Delta t)^2+C_2(T)h^4,
\end{align}
provided that $\Delta t=O(h^2)$ when $\theta=0$ and $\Delta t>0$ when $\theta=1$, where
	\begin{align*}
		&\hat{Z}_p^{n+\theta}=\kappa_1\Theta_{\xi}^{n+1}+\kappa_2\Theta_{\eta}^{n+\theta}+\Lambda_p^{n+\theta}-\hat{\Lambda}_p^{n+\theta},\\
		&R_h^{n+\theta}=-\frac{1}{\Delta t}\int_{t_{n-1+\theta}}^{t_{n+\theta}}(s-t_n)\eta_{tt}(s)\mathrm{ds},
	\end{align*}
	and
	\begin{align*}
		C_1(T)&=C\left(
		\|\eta_t\|^2_{L^2((0,T);L^2(\Omega))}
		+\|\eta_{tt}\|^2_{L^2((0,T);H^{-1}(\Omega))}\right),\nonumber\\
		C_2(T)&=C\left[
		\|\xi_t\|^2_{L^2((0,T);H^2(\Omega))}
		+\|\eta_t\|^2_{L^2((0,T);H^2(\Omega))}
		+\|\nabla\mathbf{u}\|^2_{L^{2}((0,T);H^2(\Omega))}\right.\nonumber\\
		&+\left.\|\xi\|^2_{L^{\infty}((0,T);H^2(\Omega))}
		+\|\eta\|^2_{L^{\infty}((0,T);H^2(\Omega))}
		+\|\nabla\mathbf{u}\|^2_{L^{\infty}((0,T);H^2(\Omega))}
		+\|\operatorname{div}\mathbf{u}_t\|^2_{L^{2}((0,T);H^2(\Omega))}
		\right].
	\end{align*}
\end{lemma}
\begin{proof}
	Subtracting \eqref{3.14} form \eqref{2.24}, \eqref{3.15} form \eqref{2.25}, \eqref{3.16} form \eqref{2.26}, respectively, we get
	\begin{align*}
		(\mathcal{N}(\varepsilon(\mathbf{u}(t_{n+1})))-
		\mathcal{N}(\varepsilon(\mathbf{u}_{h}^{n+1})), \varepsilon(\mathbf{v}_h) )
		-(E_{\xi}^{n+1}, \operatorname{div} \mathbf{v}_h )=0,&\quad\forall{\mathbf{v}_{h}}\in{\mathbf{V}_{h}},	\\
		\kappa_{3}(E_{\xi}^{n+1},{\varphi_{h}})
		+(\operatorname{div}E_{\mathbf{u}}^{n+1},{\varphi_{h}})=
		k_{1}(E_{\eta}^{n+\theta},{\varphi_{h}})
		+(1-\theta)\kappa_1\Delta t(d_t\eta(t_{n+1}),\varphi_h)
		,&\quad\forall \varphi_{h}\in M_{h},\\
		(d_t E_{\eta}^{n+1},\psi_h)+\frac{1}{\mu_{f}}({K(\nabla(\kappa_{1}E_{\xi}^{n+1} +\kappa_{2}E_{\eta}^{n+1}))},{\nabla\psi_{h}})=(R_h^{n+1},\psi_h),&\quad\forall \psi_{h}\in M_{h}.
	\end{align*}	
	Using the definitions of the projection $\mathcal{R}_h$, $\mathcal{Q}_h$ and $\mathcal{S}_h$, we have
	\begin{align}
		(\mathcal{N}(\varepsilon(\mathbf{u}(t_{n+1})))-
		\mathcal{N}(\varepsilon(\mathbf{u}_{h}^{n+1})), \varepsilon(\mathbf{v}_h) )
		-(\Lambda_{\xi}^{n+1}+\Theta_{\xi}^{n+1}, \operatorname{div} \mathbf{v}_h )=0,&\quad\forall{\mathbf{v}_{h}}\in{\mathbf{V}_{h}},\label{4.10}	\\
		\kappa_{3}(\Theta_{\xi}^{n+1},{\varphi_{h}})
		+(\operatorname{div}\Lambda_{\mathbf{u}}^{n+1},{\varphi_{h}})
		+(\operatorname{div}\Theta_{\mathbf{u}}^{n+1},{\varphi_{h}})\qquad\qquad\qquad\qquad~		
		&\nonumber\\
		=k_{1}(\Theta_{\eta}^{n+\theta},{\varphi_{h}})
		+(1-\theta)\kappa_1\Delta t(d_t\eta(t_{n+1}),\varphi_h)
		,&\quad\forall \varphi_{h}\in M_{h},\label{4.11}\\
		(d_t \Theta_{\eta}^{n+1},\psi_h)+\frac{1}{\mu_{f}}({K(\nabla(\kappa_{1}\hat{\Theta}_{\xi}^{n+1} +\kappa_{2}\hat{\Theta}_{\eta}^{n+1}))},{\nabla\psi_{h}})=(R_h^{n+1},\psi_h),&\quad\forall \psi_{h}\in M_{h}.\label{4.12}
	\end{align}	
	Setting $\mathbf{v}_h=\Theta_{\mathbf{u}}^{n}=\mathcal{R}_h\mathbf{u}(t_{n})-\mathbf{u}_h^{n}$ after applying the difference operator $d_t$ to the equation \eqref{4.10}, we get
	\begin{align}\label{4.13}
		&(d_t(\mathcal{N}(\varepsilon(\mathcal{R}_h\mathbf{u}(t_{n+1})))
		-\mathcal{N}(\varepsilon(\mathbf{u}_{h}^{n+1}))), \varepsilon(\mathcal{R}_h\mathbf{u}(t_{n}))-\varepsilon(\mathbf{u}_h^{n}) )
		\nonumber\\
		&+(d_t(\mathcal{N}(\varepsilon(\mathbf{u}(t_{n+1})))-
		\mathcal{N}(\varepsilon(\mathcal{R}_h\mathbf{u}(t_{n+1})))), \varepsilon(\mathcal{R}_h\mathbf{u}(t_{n}))-\varepsilon(\mathbf{u}_h^{n}) )
		\nonumber\\
		&=(d_t\Lambda_{\xi}^{n+1}, \operatorname{div}\Theta_{\mathbf{u}}^{n})
		+(d_t\Theta_{\xi}^{n+1}, \operatorname{div}\Theta_{\mathbf{u}}^{n}).
	\end{align}	
	Setting $\mathbf{v}_h=d_t\Theta_{\mathbf{u}}^{n+1}=d_t(\mathcal{R}_h\mathbf{u}(t_{n+1})-\mathbf{u}_h^{n+1})$ in \eqref{4.10}, we get
	\begin{align}\label{4.14}
		&(\mathcal{N}(\varepsilon(\mathcal{R}_h\mathbf{u}(t_{n+1})))-
		\mathcal{N}(\varepsilon(\mathbf{u}_{h}^{n+1})), d_t(\varepsilon(\mathcal{R}_h\mathbf{u}(t_{n+1}))-\varepsilon(\mathbf{u}_h^{n+1}) ))\nonumber\\
		&+
		(\mathcal{N}(\varepsilon(\mathbf{u}(t_{n+1})))-
		\mathcal{N}(\varepsilon(\mathcal{R}_h\mathbf{u}(t_{n+1}))), d_t(\varepsilon(\mathcal{R}_h\mathbf{u}(t_{n+1}))-\varepsilon(\mathbf{u}_h^{n+1}) ))
		\nonumber\\
		&=(\Lambda_{\xi}^{n+1},d_t \operatorname{div}\Theta_{\mathbf{u}}^{n+1})
		+(\Theta_{\xi}^{n+1},d_t \operatorname{div}\Theta_{\mathbf{u}}^{n+1}).
	\end{align}	
	Adding \eqref{4.13}-\eqref{4.14} and using the fact that
	\begin{align}\label{4.15}
		&d_t(\mathcal{N}(\varepsilon(\mathcal{R}_h\mathbf{u}(t_{n+1})))-\mathcal{N}(\varepsilon(\mathbf{u}_{h}^{n+1})), \varepsilon(\mathcal{R}_h\mathbf{u}(t_{n+1}))-\varepsilon(\mathbf{u}_h^{n+1}) )\nonumber\\
		&=(d_t(\mathcal{N}(\varepsilon(\mathcal{R}_h\mathbf{u}(t_{n+1})))-\mathcal{N}(\varepsilon(\mathbf{u}_{h}^{n+1}))), \varepsilon(\mathcal{R}_h\mathbf{u}(t_{n}))-\varepsilon(\mathbf{u}_h^{n}) )\nonumber\\
		&+(\mathcal{N}(\varepsilon(\mathcal{R}_h\mathbf{u}(t_{n+1})))-\mathcal{N}(\varepsilon(\mathbf{u}_{h}^{n+1})),d_t( \varepsilon(\mathcal{R}_h\mathbf{u}(t_{n+1}))-\varepsilon(\mathbf{u}_h^{n+1}) )),
	\end{align}
	we get
	\begin{align}\label{4.16}
		&d_t(\mathcal{N}(\varepsilon(\mathcal{R}_h\mathbf{u}(t_{n+1})))-\mathcal{N}(\varepsilon(\mathbf{u}_{h}^{n+1})),\varepsilon(\mathcal{R}_h\mathbf{u}(t_{n+1}))-\varepsilon(\mathbf{u}_h^{n+1}) )\nonumber\\
		&=(d_t\Lambda_{\xi}^{n+1}, \operatorname{div}\Theta_{\mathbf{u}}^{n})
		+(d_t\Theta_{\xi}^{n+1}, \operatorname{div}\Theta_{\mathbf{u}}^{n})
		+(\Lambda_{\xi}^{n+1},d_t \operatorname{div}\Theta_{\mathbf{u}}^{n+1})
		+(\Theta_{\xi}^{n+1},d_t \operatorname{div}\Theta_{\mathbf{u}}^{n+1})\nonumber\\
		&-(d_t(\mathcal{N}(\varepsilon(\mathbf{u}(t_{n+1})))-
		\mathcal{N}(\varepsilon(\mathcal{R}_h\mathbf{u}(t_{n+1})))
		), \varepsilon(\mathcal{R}_h\mathbf{u}(t_{n}))-\varepsilon(\mathbf{u}_h^{n}) )\nonumber\\
		&-(\mathcal{N}(\varepsilon(\mathbf{u}(t_{n+1})))-
		\mathcal{N}(\varepsilon(\mathcal{R}_h\mathbf{u}(t_{n+1}))), d_t(\varepsilon(\mathcal{R}_h\mathbf{u}(t_{n+1}))-\varepsilon(\mathbf{u}_h^{n+1}) ))\nonumber\\
		&=(d_t\Lambda_{\xi}^{n+1}, \operatorname{div}\Theta_{\mathbf{u}}^{n})
		+(d_t\Theta_{\xi}^{n+1}, \operatorname{div}\Theta_{\mathbf{u}}^{n})
		+(\Lambda_{\xi}^{n+1},d_t \operatorname{div}\Theta_{\mathbf{u}}^{n+1})
		+(\Theta_{\xi}^{n+1},d_t \operatorname{div}\Theta_{\mathbf{u}}^{n+1})\nonumber\\
		&-d_t(\mathcal{N}(\varepsilon(\mathbf{u}(t_{n+1})))-
		\mathcal{N}(\varepsilon(\mathcal{R}_h\mathbf{u}(t_{n+1}))
		), \varepsilon(\mathcal{R}_h\mathbf{u}(t_{n+1}))-\varepsilon(\mathbf{u}_h^{n+1}) ).
	\end{align}
	Setting $\varphi_h=\Theta_{\xi}^{n+1}$ after applying the difference operator $d_t$ to the equation \eqref{4.11}, we get
	\begin{align}\label{4.17}
		\kappa_{3}(d_t\Theta_{\xi}^{n+1},\Theta_{\xi}^{n+1})
		+(d_t\operatorname{div}\Lambda_{\mathbf{u}}^{n+1},\Theta_{\xi}^{n+1})
		+(d_t\operatorname{div}\Theta_{\mathbf{u}}^{n+1},\Theta_{\xi}^{n+1})\qquad\qquad\qquad\qquad~	
		&\nonumber\\
		=k_{1}(d_t\Theta_{\eta}^{n+\theta},\Theta_{\xi}^{n+1})
		+(1-\theta)\kappa_1 (d_t\eta(t_{n+1})-d_t\eta(t_{n}),\Theta_{\xi}^{n+1}).
	\end{align}
	When $\theta=1$, setting $\psi_h=\hat{Z}_p^{n+1}=\kappa_1\hat{\Theta}_{\xi}^{n+1}+\kappa_2\hat{\Theta}_{\eta}^{n+1}=\kappa_1\Theta_{\xi}^{n+1}+\kappa_2\Theta_{\eta}^{n+1}+\Lambda_p^{n+1}-\hat{\Lambda}_p^{n+1} $ in \eqref{4.12}, we get
	\begin{align}\label{4.18}
		(d_t \Theta_{\eta}^{n+1},\kappa_1\Theta_{\xi}^{n+1}+\kappa_2\Theta_{\eta}^{n+1})
		&+\frac{1}{\mu_{f}}({K
			\nabla\hat{Z}_p^{n+1}},{\nabla
			\hat{Z}_p^{n+1}})\nonumber\\
		&=-(d_t \Theta_{\eta}^{n+1},\Lambda_p^{n+1}-\hat{\Lambda}_p^{n+1})+(R_h^{n+1},\hat{Z}_p^{n+1}).
	\end{align}
	Adding \eqref{4.16}-\eqref{4.18}, we get 
	\begin{align}\label{125-2}
		&d_t(\mathcal{N}(\varepsilon(\mathcal{R}_h\mathbf{u}(t_{n+1})))-\mathcal{N}(\varepsilon(\mathbf{u}_{h}^{n+1})),\varepsilon(\mathcal{R}_h\mathbf{u}(t_{n+1}))-\varepsilon(\mathbf{u}_h^{n+1}) )
		+\frac{\kappa_3}{2}d_t\|\Theta_{\xi}^{n+1}\|^2_{L^2(\Omega)}
		\nonumber\\
		&+\frac{\kappa_3\Delta t}{2}\|d_t\Theta_{\xi}^{n+1}\|^2_{L^2(\Omega)}
		+\frac{\kappa_2}{2}d_t\|\Theta_{\eta}^{n+1}\|^2_{L^2(\Omega)}
		+\frac{\kappa_2\Delta t}{2}\|d_t\Theta_{\eta}^{n+1}\|^2_{L^2(\Omega)}
		+(d_t\operatorname{div}\Lambda_{\mathbf{u}}^{n+1},\Theta_{\xi}^{n+1})\nonumber\\
		&+\frac{1}{\mu_{f}}(K\nabla\hat{Z}_p^{n+1},\nabla\hat{Z}_p^{n+1})+(d_t \Theta_{\eta}^{n+1},\Lambda_p^{n+1}-\hat{\Lambda}_p^{n+1})\nonumber\\
		&=(d_t\Lambda_{\xi}^{n+1}, \operatorname{div}\Theta_{\mathbf{u}}^{n})
		+(d_t\Theta_{\xi}^{n+1}, \operatorname{div}\Theta_{\mathbf{u}}^{n})
		+(\Lambda_{\xi}^{n+1},d_t \operatorname{div}\Theta_{\mathbf{u}}^{n+1})
		+(R_h^{n+1},\hat{Z}_p^{n+1})\nonumber\\
		&-d_t(\mathcal{N}(\varepsilon(\mathbf{u}(t_{n+1})))-
		\mathcal{N}(\varepsilon(\mathcal{R}_h\mathbf{u}(t_{n+1}))
		), \varepsilon(\mathcal{R}_h\mathbf{u}(t_{n+1}))-\varepsilon(\mathbf{u}_h^{n+1}) ).
	\end{align}
	When $\theta=0$, setting $\psi_h=\hat{Z}_p^{n}=\kappa_1\Theta_{\xi}^{n+1}+\kappa_2\Theta_{\eta}^{n}+\Lambda_p^{n}-\hat{\Lambda}_p^{n}$ in \eqref{4.12} after lowing the degree from $n+1$ to $n$, we get
	\begin{align*}
		&(d_t\Theta_{\eta}^{n},\kappa_1\Theta_{\xi}^{n+1}+\kappa_2\Theta_{\eta}^{n})
		+(d_t\Theta_{\eta}^{n+1},\Lambda_p^{n}-\hat{\Lambda}_p^{n})\nonumber\\
		&+\frac{1}{\mu_{f}}({K(\nabla(\kappa_{1}\hat{\Theta}_{\xi}^{n}+\kappa_{2}\Lambda_{\eta}^{n}+\kappa_{2}\Theta_{\eta}^{n}))},{\nabla\hat{Z}_p^{n}})
		=(R_h^{n+1},\hat{Z}_p^{n}).
	\end{align*}
	Using the properties of projections $\mathcal{Q}_h$ and $\mathcal{S}_h$, we get
	\begin{align}\label{125-1}
		&(d_t\Theta_{\eta}^{n},\kappa_1\Theta_{\xi}^{n+1}+\kappa_2\Theta_{\eta}^{n})
		+(d_t\Theta_{\eta}^{n+1},\Lambda_p^{n}-\hat{\Lambda}_p^{n})\nonumber\\
		&+\frac{1}{\mu_{f}}({K\nabla\hat{Z}_p^{n}},{\nabla\hat{Z}_p^{n}})-\frac{\kappa_1\Delta t}{\mu_{f}}({Kd_t\nabla\Theta_{\xi}^{n+1}},{\nabla\hat{Z}_p^{n}})=(R_h^{n+1},\hat{Z}_p^{n}).
	\end{align}
	Adding \eqref{4.16}, \eqref{4.17} and \eqref{125-1}, we get 
	\begin{align}\label{125-3}
		&d_t(\mathcal{N}(\varepsilon(\mathcal{R}_h\mathbf{u}(t_{n+1})))-\mathcal{N}(\varepsilon(\mathbf{u}_{h}^{n+1})),\varepsilon(\mathcal{R}_h\mathbf{u}(t_{n+1}))-\varepsilon(\mathbf{u}_h^{n+1}) )
		+\frac{\kappa_3}{2}d_t\|\Theta_{\xi}^{n+1}\|^2_{L^2(\Omega)}
		\nonumber\\
		&+\frac{\kappa_3\Delta t}{2}\|d_t\Theta_{\xi}^{n+1}\|^2_{L^2(\Omega)}
		+\frac{\kappa_2}{2}d_t\|\Theta_{\eta}^{n}\|^2_{L^2(\Omega)}+\frac{\kappa_2\Delta t}{2}\|d_t\Theta_{\eta}^{n}\|^2_{L^2(\Omega)}
		+(d_t\operatorname{div}\Lambda_{\mathbf{u}}^{n+1},\Theta_{\xi}^{n+1})
		\nonumber\\
		&+\frac{1}{\mu_{f}}({K\nabla \hat{Z}_p^{n}},{\nabla\hat{Z}_p^{n}})
		-\frac{\kappa_1\Delta t}{\mu_{f}}({K
		d_t\nabla\Theta_{\xi}^{n+1}},{\nabla\hat{Z}_p^{n}})+(d_t\Theta_{\eta}^{n},\Lambda_p^{n}-\hat{\Lambda}_p^{n})\nonumber\\
		&=(d_t\Lambda_{\xi}^{n+1}, \operatorname{div}\Theta_{\mathbf{u}}^{n})
		+(d_t\Theta_{\xi}^{n+1}, \operatorname{div}\Theta_{\mathbf{u}}^{n})
		+(\Lambda_{\xi}^{n+1},d_t \operatorname{div}\Theta_{\mathbf{u}}^{n+1})
		+(R_h^{n},\hat{Z}_p^{n})\nonumber\\
		&-d_t(\mathcal{N}(\varepsilon(\mathbf{u}(t_{n+1})))-\mathcal{N}(\varepsilon(\mathcal{R}_h\mathbf{u}(t_{n+1}))),\varepsilon(\mathcal{R}_h\mathbf{u}(t_{n+1}))-\varepsilon(\mathbf{u}_h^{n+1}) )\nonumber\\
		&+\kappa_1(d_t\eta(t_{n+1})-d_t\eta(t_{n}),\Theta_{\xi}^{n+1}).
	\end{align}
	Combining \eqref{125-2} and \eqref{125-3} into one  with $\theta$, we get
	\begin{align}\label{125-4}
		&d_t(\mathcal{N}(\varepsilon(\mathcal{R}_h\mathbf{u}(t_{n+1})))-\mathcal{N}(\varepsilon(\mathbf{u}_{h}^{n+1})),\varepsilon(\mathcal{R}_h\mathbf{u}(t_{n+1}))-\varepsilon(\mathbf{u}_h^{n+1}) )
		+\frac{\kappa_3}{2}d_t\|\Theta_{\xi}^{n+1}\|^2_{L^2(\Omega)}
		\nonumber\\
		&+\frac{\kappa_3\Delta t}{2}\|d_t\Theta_{\xi}^{n+1}\|^2_{L^2(\Omega)}
		+\frac{\kappa_2}{2}d_t\|\Theta_{\eta}^{n+\theta}\|^2_{L^2(\Omega)}+\frac{\kappa_2\Delta t}{2}\|d_t\Theta_{\eta}^{n+\theta}\|^2_{L^2(\Omega)}+(d_t\operatorname{div}\Lambda_{\mathbf{u}}^{n+1},\Theta_{\xi}^{n+1})\nonumber\\
		&+\frac{1}{\mu_{f}}({K\nabla \hat{Z}_p^{n+\theta}},{\nabla\hat{Z}_p^{n+\theta}})-(1-\theta)\frac{\kappa_1\Delta t}{\mu_{f}}({K
		d_t\nabla\Theta_{\xi}^{n+1}},{\nabla\hat{Z}_p^{n+\theta}})+(d_t\Theta_{\eta}^{n+\theta},\Lambda_p^{n+\theta}-\hat{\Lambda}_p^{n+\theta})\nonumber\\
		&=
		(d_t\Lambda_{\xi}^{n+1}, \operatorname{div}\Theta_{\mathbf{u}}^{n})
		+(d_t\Theta_{\xi}^{n+1}, \operatorname{div}\Theta_{\mathbf{u}}^{n})
		+(\Lambda_{\xi}^{n+1},d_t \operatorname{div}\Theta_{\mathbf{u}}^{n+1})
		+(R_h^{n+\theta},\hat{Z}_p^{n+\theta})\nonumber\\
		&-d_t(\mathcal{N}(\varepsilon(\mathbf{u}(t_{n+1})))-\mathcal{N}(\varepsilon(\mathcal{R}_h\mathbf{u}(t_{n+1}))
		),\varepsilon(\mathcal{R}_h\mathbf{u}(t_{n+1}))-\varepsilon(\mathbf{u}_h^{n+1}) )\nonumber\\
		&+(1-\theta)\kappa_1 (d_t\eta(t_{n+1})-d_t\eta(t_{n}),\Theta_{\xi}^{n+1}).
	\end{align}
	Applying the summation operator $\Delta t\sum_{n=0}^{l}$ to both sides of \eqref{125-4},
	we get
	\begin{align}\label{4.24}
		&\tilde{C}_3\|\varepsilon(\Theta_{\mathbf{u}}^{l+1})\|^2_{L^2(\Omega)}
		+\frac{\kappa_3}{2}\|\Theta_{\xi}^{l+1}\|^2_{L^2(\Omega)}
		+\frac{\kappa_2}{2}\|\Theta_{\eta}^{l+\theta}\|^2_{L^2(\Omega)}
		+\Delta t\sum_{n=0}^{l}\frac{K_1}{\mu_{f}}\|\nabla\hat{Z}_p^{n+\theta}\|^2_{L^2(\Omega)}
		\nonumber\\
		&+\Delta t\sum_{n=0}^{l}\left[
		\frac{\kappa_3\Delta t}{2}\|d_t\Theta_{\xi}^{n+1}\|^2_{L^2(\Omega)}
		+\frac{\kappa_2\Delta t}{2}\|d_t\Theta_{\eta}^{n+\theta}\|^2_{L^2(\Omega)}\right]
		\nonumber\\
		&\leq
		\tilde{C}_3\|\varepsilon(\Theta_{\mathbf{u}}^{0})\|^2_{L^2(\Omega)}
		+\frac{\kappa_3}{2}\|\Theta_{\xi}^{0}\|^2_{L^2(\Omega)}
		+\frac{\kappa_2}{2}d_t\|\Theta_{\eta}^{\theta-1}\|^2_{L^2(\Omega)}+\sum_{j=1}^{7}\Phi_j,
	\end{align}
	where
	\begin{align*}
		\Phi_1&=\Delta t\sum_{n=0}^{l}
		(R_h^{n+\theta},\hat{Z}_p^{n+\theta}),\nonumber\\
		\Phi_2&=\Delta t\sum_{n=0}^{l}
		(d_t\operatorname{div}\Lambda_{\mathbf{u}}^{n+1},\Theta_{\xi}^{n+1}),\nonumber\\
		\Phi_3&=-\Delta t\sum_{n=0}^{l}(d_t\Theta_{\eta}^{n+\theta},\Lambda_p^{n+\theta}-\hat{\Lambda}_p^{n+\theta}),\nonumber\\
		\Phi_4&=(1-\theta)\Delta t\sum_{n=0}^{l}\left[\frac{\kappa_1\Delta t}{\mu_{f}}({K
		d_t\nabla\Theta_{\xi}^{n+1}},{\nabla\hat{Z}_p^{n+\theta}})\right],\nonumber\\
		\Phi_5&=(1-\theta)\Delta t\sum_{n=0}^{l}\left[\kappa_1 (d_t\eta(t_{n+1})-d_t\eta(t_{n}),\Theta_{\xi}^{n+1})\right],\nonumber\\
		\Phi_6&=\Delta t\sum_{n=0}^{l}\left[
		(d_t\Lambda_{\xi}^{n+1}, \operatorname{div}\Theta_{\mathbf{u}}^{n})
		+(d_t\Theta_{\xi}^{n+1}, \operatorname{div}\Theta_{\mathbf{u}}^{n})
		+(\Lambda_{\xi}^{n+1},d_t \operatorname{div}\Theta_{\mathbf{u}}^{n+1})
		\right],\nonumber\\
		\Phi_7&=-\Delta t\sum_{n=0}^{l}d_t
		(\mathcal{N}(\varepsilon(\mathbf{u}(t_{n+1})))-
		\mathcal{N}(\varepsilon(\mathcal{R}_h\mathbf{u}(t_{n+1}))),\varepsilon(\mathcal{R}_h\mathbf{u}(t_{n+1}))-\varepsilon(\mathbf{u}_h^{n+1}) ).
	\end{align*}
	Using Cauchy-Schwarz inequality, we get
	\begin{align}\label{4.25}
		\Phi_1&\leq \Delta t\sum_{n=0}^{l}\left(
		\|R_h^{n+\theta}\|_{H^{-1}(\Omega)}
		\|\nabla\hat{Z}_p^{n+\theta}\|_{L^2(\Omega)}
		\right)\nonumber\\
		&\leq\Delta t\sum_{n=0}^{l}\left[\frac{K}{4\mu_f}\|\nabla\hat{Z}_p^{n+\theta}\|^2_{L^2(\Omega)}
		+\frac{\mu_f}{K}\|R_h^{n+\theta}\|^2_{H^{-1}(\Omega)}\right]\nonumber\\
		&\leq\Delta t\sum_{n=0}^{l}\frac{K}{4\mu_f}\|\nabla\hat{Z}_p^{n+\theta}\|^2_{L^2(\Omega)}+\Delta t\sum_{n=0}^{l}\frac{\mu_f\Delta t}{3K}\|\eta_{tt}\|^2_{L^2((t_{n-1+\theta},t_{n+\theta});H^{-1}(\Omega))}\nonumber\\
		&\leq\Delta t\sum_{n=0}^{l}\frac{K}{4\mu_f}\|\nabla\hat{Z}_p^{n+\theta}\|^2_{L^2(\Omega)}
		+\frac{\mu_f(\Delta t)^2}{3K}\|\eta_{tt}\|^2_{L^2((0,t_{l+1});H^{-1}(\Omega))},
	\end{align}
	where we have used the fact that
	\begin{align*}
		\|R_h^{n+\theta}\|^2_{H^{-1}(\Omega)}\leq
		\frac{\Delta t}{3}\int_{t_{n-1+\theta}}^{t_{n+\theta}}\|
		\eta_{tt}\|^2_{H^{-1}(\Omega)}\mathrm{dt}.
	\end{align*}	
		Using \eqref{3.13}, \eqref{4.10}, Cauchy-Schwarz inequality and \eqref{4.21}, we get
	\begin{align}\label{4.23}
		\|\Theta_{\xi}^{n+1}\|_{L^2(\Omega)}
		&\leq \frac{1}{\beta}\sup_{\mathbf{0}\neq\mathbf{v}_h\in \mathbf{V}_h}\frac{(\Theta_{\xi}^{n+1}, \div \mathbf{v}_h)}{\|\mathbf{v}_h\|_{H^1(\Omega)}}\nonumber\\
		&\leq \frac{1}{\beta}\sup_{\mathbf{0}\neq\mathbf{v}_h\in \mathbf{V}_h}\frac{(\mathcal{N}(\varepsilon(\mathbf{u}(t_{n+1})))-
		\mathcal{N}(\varepsilon(\mathbf{u}_{h}^{n+1})), \varepsilon(\mathbf{v}_h) )
		-(\Lambda_{\xi}^{n+1}, \operatorname{div} \mathbf{v}_h )}{\|\mathbf{v}_h\|_{H^1(\Omega)}}\nonumber\\
		&\leq \frac{1}{\beta}\left[
		\|\mathcal{N}(\varepsilon(\mathbf{u}(t_{n+1})))-
		\mathcal{N}(\varepsilon(\mathbf{u}_{h}^{n+1}))\|
		_{L^2(\Omega)}+\|\Lambda_{\xi}^{n+1}\|_{L^2(\Omega)}\right]\nonumber\\
		&\leq \frac{\tilde{C}_2}{\beta}\left[
		\|\varepsilon(\mathbf{u}(t_{n+1}))-
		\varepsilon(\mathbf{u}_{h}^{n+1})\|
		_{L^2(\Omega)}+\|\Lambda_{\xi}^{n+1}\|_{L^2(\Omega)}\right]\nonumber\\
		&\leq \frac{\tilde{C}_2}{\beta}\left[
		\|\varepsilon(\Lambda_{\mathbf{u}}^{n+1})\|
		_{L^2(\Omega)}+\|\varepsilon(\Theta_{\mathbf{u}}^{n+1})\|
		_{L^2(\Omega)}+\|\Lambda_{\xi}^{n+1}\|_{L^2(\Omega)}\right].
	\end{align}	
	Using Cauchy-Schwarz inequality, Young inequality and \eqref{4.23}, we get
	\begin{align}
		\Phi_2&\leq \Delta t\sum_{n=0}^{l}\left[
		\frac{1}{2}\|d_t\operatorname{div}\Lambda_{\mathbf{u}}^{n+1}\|^2_{L^2(\Omega)}+\frac{1}{2}\|\Theta_{\xi}^{n+1}\|^2_{L^2(\Omega)}\right]\nonumber\\
		&\leq \Delta t\sum_{n=0}^{l}
		\frac{\tilde{C}_2}{2\beta}\left[ 
		\|\varepsilon(\Lambda_{\mathbf{u}}^{n+1})\|
		_{L^2(\Omega)}+\|\varepsilon(\Theta_{\mathbf{u}}^{n+1})\|
		_{L^2(\Omega)}+\|\Lambda_{\xi}^{n+1}\|_{L^2(\Omega)}\right]\nonumber\\
		&+\Delta t\sum_{n=0}^{l}
		\frac{1}{2}\|d_t\operatorname{div}\Lambda_{\mathbf{u}}^{n+1}\|^2_{L^2(\Omega)}.
	\end{align}	
	Using Cauchy-Schwarz inequality and Young inequality, we get
	\begin{align}\label{4.26}
		\Phi_3&=-\Delta t\sum_{n=0}^{l}(d_t\Theta_{\eta}^{n+\theta},\Lambda_p^{n+\theta}-\hat{\Lambda}_p^{n+\theta})\nonumber\\
		&=(\Theta_{\eta}^{l+\theta},\Lambda_p^{l+\theta}-\hat{\Lambda}_p^{l+\theta})-\Delta t\sum_{n=1}^{l}(\Theta_{\eta}^{n-1+\theta},d_t\Lambda_p^{n+\theta}-d_t\hat{\Lambda}_p^{n+\theta})\nonumber\\
		&\leq\frac{\Delta t}{2} \sum_{n=1}^{l}\|\Theta_{\eta}^{n-1+\theta}\|^2_{L^2(\Omega)}
		+\frac{\Delta t}{2}\sum_{n=1}^{l}\|d_t\Lambda_p^{n+\theta}\|^2_{L^2(\Omega)}
		++\frac{\Delta t}{2}\sum_{n=1}^{l}\|d_t\hat{\Lambda}_p^{n+\theta})\|^2_{L^2(\Omega)}
		\nonumber\\
		&+\frac{\kappa_2}{4}\|\Theta_{\eta}^{l+\theta}\|^2_{L^2(\Omega)}+\frac{1}{\kappa_2}\|\Lambda_p^{l+\theta}\|^2_{L^2(\Omega)}+\frac{1}{\kappa_2}\|\hat{\Lambda}_p^{l+\theta}\|^2_{L^2(\Omega)}.
	\end{align}	
	Using Cauchy-Schwarz inequality, Young inequality and inverse inequality \eqref{122-8}, we get
	\begin{align}\label{4.27}
		\Phi_4&=(1-\theta)\Delta t\sum_{n=0}^{l}\frac{\kappa_1\Delta t}{\mu_{f}}({Kd_t\nabla\Theta_{\xi}^{n+1}},{\nabla\hat{Z}_p^{n+\theta}})\nonumber\\
		&\leq (1-\theta)\Delta t\sum_{n=0}^{l}
		\frac{K_2\kappa_1\Delta t}{\mu_{f}}\|d_t\nabla\Theta_{\xi}^{n+1}\|_{L^2(\Omega)}\|\nabla\hat{Z}_p^{n+\theta}\|_{L^2(\Omega)}\nonumber\\
		&\leq (1-\theta)\Delta t\sum_{n=0}^{l}
		\frac{K_2^2\kappa_1^2(\Delta t)^2}{\mu_{f}K_1}\|d_t\nabla\Theta_{\xi}^{n+1}\|^2_{L^2(\Omega)}
		+(1-\theta)\Delta t\sum_{n=0}^{l}\frac{K_1}{4\mu_f}\|\nabla\hat{Z}_p^{n+\theta}\|_{L^2(\Omega)}\nonumber\\
		&\leq (1-\theta)\Delta t\sum_{n=0}^{l}
		\frac{cK_2^2\kappa_1^2(\Delta t)^2}{\mu_{f}K_1h^2}\|d_t\Theta_{\xi}^{n+1}\|^2_{L^2(\Omega)}
		+(1-\theta)\Delta t\sum_{n=0}^{l}\frac{K_1}{4\mu_f}\|\nabla\hat{Z}_p^{n+\theta}\|_{L^2(\Omega)}\nonumber\\
		&\leq (1-\theta)\Delta t\sum_{n=0}^{l}
		\frac{\kappa_3\Delta t}{4}\|d_t\Theta_{\xi}^{n+1}\|^2_{L^2(\Omega)}
		+(1-\theta)\Delta t\sum_{n=0}^{l}\frac{K_1}{4\mu_f}\|\nabla\hat{Z}_p^{n+\theta}\|_{L^2(\Omega)},
	\end{align}	
	where we assume $\Delta t\leq\frac{\mu_fK_1\kappa_3}{cK_2^2\kappa_1^2}h^2$.	
	Using the integration by parts and $d_t\eta(t_0)=0$, we get
	\begin{align}\label{4.28}
		\Phi_5&=(1-\theta)\kappa_1\Delta t\sum_{n=0}^{l}[ (d_t\eta(t_{n+1})-d_t\eta(t_{n}),\Theta_{\xi}^{n+1})]\nonumber\\
		&=(1-\theta)\kappa_1(\Delta t)^2\left[
		\frac{1}{\Delta t}(d_t\eta(t_{l+1}),\Theta_{\xi}^{l+1})
		-\sum_{n=1}^{l}(d_t\eta(t_{n+1}), d_t\Theta_{\xi}^{n+1})		
		\right]\nonumber\\
		&\leq(1-\theta)\left((\Delta t)^2\frac{\kappa_1^2}{\kappa_3}\|d_t\eta(t_{l+1})\|^2_{L^2(\Omega)}
		+\frac{\kappa_3}{4}\|\Theta_{\xi}^{l+1}\|^2_{L^2(\Omega)}\right)\nonumber\\
		&+(1-\theta)(\Delta t)^2\sum_{n=1}^{l}
		\left[
		\frac{\kappa_1^2}{\kappa_3}\|d_t\eta(t_{n+1})\|^2_{L^2(\Omega)}
		+\frac{\kappa_3}{4}\|d_t\Theta_{\xi}^{l+1}\|^2_{L^2(\Omega)}
		\right]\nonumber\\
		&\leq(1-\theta)(\Delta t)^2\frac{2\kappa_1^2}{\kappa_3}\|\eta_t\|^2_{L^2((0,t_{l+1});L^2(\Omega))}
		+(1-\theta)\frac{\kappa_3}{4}\|\Theta_{\xi}^{l+1}\|^2_{L^2(\Omega)}\nonumber\\
		&+(1-\theta)\Delta t\sum_{n=1}^{l}\frac{\kappa_3}{4}\Delta t
		\|d_t\Theta_{\xi}^{l+1}\|^2_{L^2(\Omega)}.
	\end{align}	
	Using Cauchy-Schwarz inequality and Young inequality, we get
	\begin{align}\label{4.29}
		\Phi_6&=\Delta t\sum_{n=0}^{l}[
		(d_t\Lambda_{\xi}^{n+1}, \operatorname{div}\Theta_{\mathbf{u}}^{n})
		+(d_t\Theta_{\xi}^{n+1}, \operatorname{div}\Theta_{\mathbf{u}}^{n})
		+(\Lambda_{\xi}^{n+1},d_t \operatorname{div}\Theta_{\mathbf{u}}^{n+1})]\nonumber\\
		&\leq 
		\Delta t\sum_{n=0}^{l}\left[
		\frac{1}{2}\|d_t\Lambda_{\xi}^{n+1}\|^2_{L^2(\Omega)}
		+(\frac{1}{2}+\frac{1}{2\epsilon_1})\|\varepsilon(\Theta_{\mathbf{u}}^{n})\|^2_{L^2(\Omega)}
		+\frac{\epsilon_1}{2}\|d_t\Theta_{\xi}^{n+1}\|^2_{L^2(\Omega)}
		\right]\nonumber\\
		&+\Delta t\sum_{n=1}^{l}\left[ 
		\frac{1}{2}\|d_t\Lambda_{\xi}^{n+1}\|^2_{L^2(\Omega)}
		+\frac{1}{2}\|\varepsilon(\Theta_{\mathbf{u}}^{n})\|^2_{L^2(\Omega)}\right]
		+\frac{1}{2\epsilon_2}\|\Lambda_{\xi}^{l+1}\|^2_{L^2(\Omega)}
		+\frac{\epsilon_2}{2}\|\varepsilon(\Theta_{\mathbf{u}}^{l+1})\|^2_{L^2(\Omega)}
		\nonumber\\
		&+\frac{1}{2}\|\Lambda_{\xi}^{1}\|^2_{L^2(\Omega)}
		+\frac{1}{2}\|\varepsilon(\Theta_{\mathbf{u}}^{0})\|^2_{L^2(\Omega)}\nonumber\\
		&\leq 
		\Delta t\sum_{n=0}^{l}\left[
		\|d_t\Lambda_{\xi}^{n+1}\|^2_{L^2(\Omega)}
		+(1+\frac{1}{2\epsilon_1})\|\varepsilon(\Theta_{\mathbf{u}}^{n})\|^2_{L^2(\Omega)}
		+\frac{\epsilon_1}{2}\|d_t\Theta_{\xi}^{n+1}\|^2_{L^2(\Omega)}
		\right]\nonumber\\
		&+\frac{1}{2\epsilon_2}\|\Lambda_{\xi}^{l+1}\|^2_{L^2(\Omega)}
		+\frac{\epsilon_2}{2}\|\varepsilon(\Theta_{\mathbf{u}}^{l+1})\|^2_{L^2(\Omega)}
		+\frac{1}{2}\|\Lambda_{\xi}^{1}\|^2_{L^2(\Omega)}
		+\frac{1}{2}\|\varepsilon(\Theta_{\mathbf{u}}^{0})\|^2_{L^2(\Omega)},
	\end{align}	
	where we have used the summation by parts formula
	\begin{align*}
		\Delta t\sum_{n=0}^{l}(\Lambda_{\xi}^{n+1},d_t \operatorname{div}\Theta_{\mathbf{u}}^{n+1})
		=
		-(\Lambda_{\xi}^{1}, \operatorname{div}\Theta_{\mathbf{u}}^{0})
		+(\Lambda_{\xi}^{l+1}, \operatorname{div}\Theta_{\mathbf{u}}^{l+1})
		-	\Delta t\sum_{n=1}^{l}(d_t\Lambda_{\xi}^{n+1}, \operatorname{div}\Theta_{\mathbf{u}}^{n}).
	\end{align*}
	Using \eqref{4.21}, we get
	\begin{align}\label{4.30}
		\Phi_7&=-\Delta t\sum_{n=0}^{l}d_t
		(\mathcal{N}(\varepsilon(\mathbf{u}(t_{n+1})))-
		\mathcal{N}(\varepsilon(\mathcal{R}_h\mathbf{u}(t_{n+1}))),\varepsilon(\mathcal{R}_h\mathbf{u}(t_{n+1}))-\varepsilon(\mathbf{u}_h^{n+1}) )\nonumber\\
		&=(\mathcal{N}(\varepsilon(\mathbf{u}(t_{l+1})))-
		\mathcal{N}(\varepsilon(\mathcal{R}_h\mathbf{u}(t_{l+1}))),\varepsilon(\mathcal{R}_h\mathbf{u}(t_{l+1}))-\varepsilon(\mathbf{u}_h^{l+1}) )\nonumber\\
		&+(\mathcal{N}(\varepsilon(\mathbf{u}(0)))-
		\mathcal{N}(\varepsilon(\mathcal{R}_h\mathbf{u}(0)),\varepsilon(\mathcal{R}_h\mathbf{u}(0))-\varepsilon(\mathbf{u}_h^{0}) )\nonumber\\
		&\leq\|\mathcal{N}(\varepsilon(\mathbf{u}(t_{l+1})))-
		\mathcal{N}(\varepsilon(\mathcal{R}_h\mathbf{u}(t_{l+1})))\|_{L^2(\Omega)}\cdot\|\varepsilon
		(\Theta_{\mathbf{u}}^{l+1})\|_{L^2(\Omega)}\nonumber\\
		&+\|\mathcal{N}(\varepsilon(\mathbf{u}(0)))-
		\mathcal{N}(\varepsilon(\mathcal{R}_h\mathbf{u}(0)))\|_{L^2(\Omega)}\cdot\|\varepsilon
		(\Theta_{\mathbf{u}}^{0})\|_{L^2(\Omega)}\nonumber\\
		&\leq 
		\|\varepsilon(\mathbf{u}(t_{l+1}))-
		\varepsilon(\mathcal{R}_h\mathbf{u}(t_{l+1}))\|_{L^2(\Omega)}\cdot\|\varepsilon
		(\Theta_{\mathbf{u}}^{l+1})\|_{L^2(\Omega)}\nonumber\\
		&+\|\varepsilon(\mathbf{u}(0))-
		\varepsilon(\mathcal{R}_h\mathbf{u}(0))\|_{L^2(\Omega)}\cdot\|\varepsilon
		(\Theta_{\mathbf{u}}^{0})\|_{L^2(\Omega)}\nonumber\\
		&=\|\varepsilon(\Lambda_{\mathbf{u}}^{l+1})\|_{L^2(\Omega)}\cdot\|\varepsilon
		(\Theta_{\mathbf{u}}^{l+1})\|_{L^2(\Omega)}
		+
		\|\varepsilon(\Lambda_{\mathbf{u}}^{0})\|_{L^2(\Omega)}\cdot\|\varepsilon
		(\Theta_{\mathbf{u}}^{0})\|_{L^2(\Omega)}\nonumber\\
		&\leq \frac{1}{\tilde{C}_3}\|\varepsilon(\Lambda_{\mathbf{u}}^{l+1})\|^2_{L^2(\Omega)}
		+\frac{\tilde{C}_3}{4}\|\varepsilon
		(\Theta_{\mathbf{u}}^{l+1})\|^2_{L^2(\Omega)}
		+\frac{1}{2}	\|\varepsilon(\Lambda_{\mathbf{u}}^{0})\|^2_{L^2(\Omega)}
		+\frac{1}{2}\|\varepsilon
		(\Theta_{\mathbf{u}}^{0})\|^2_{L^2(\Omega)}.
	\end{align}	
	Substituting \eqref{4.25}-\eqref{4.30} into \eqref{4.24}, taking $\epsilon_1,\epsilon_2$ small enough, using \eqref{4.2}, \eqref{4.5}, \eqref{4.7} and applying the discrete Gronwall inequality, we obtain
	\begin{align}
		&\tilde{C}_3\|\varepsilon(\Theta_{\mathbf{u}}^{l+1})\|^2_{L^2(\Omega)}
		+\frac{\kappa_3}{2}\|\Theta_{\xi}^{l+1}\|^2_{L^2(\Omega)}
		+\frac{\kappa_2}{2}\|\Theta_{\eta}^{l+\theta}\|^2_{L^2(\Omega)}
		+\Delta t\sum_{n=0}^{l}\frac{K_1}{\mu_{f}}\|\nabla\hat{Z}_p^{n+\theta}\|^2_{L^2(\Omega)}\nonumber\\
		&\leq C\left[ \tilde{C}_3\|\varepsilon(\Theta_{\mathbf{u}}^{0})\|^2_{L^2(\Omega)}
		+\frac{1}{2}	\|\varepsilon(\Lambda_{\mathbf{u}}^{0})\|^2_{L^2(\Omega)}
		+\frac{\kappa_3}{2}\|\Theta_{\xi}^{0}\|^2_{L^2(\Omega)}
		+\frac{\kappa_2}{2}d_t\|\Theta_{\eta}^{\theta-1}\|^2_{L^2(\Omega)}
		+\frac{1}{2}\|\Lambda_{\xi}^{1}\|^2_{L^2(\Omega)}
		\right.\nonumber\\
		&+\frac{\Delta t}{2}\sum_{n=1}^{l}\|d_t\Lambda_p^{n+\theta}\|^2_{L^2(\Omega)}
		+\frac{\Delta t}{2}\sum_{n=1}^{l}\|d_t\hat{\Lambda}_p^{n+\theta})\|^2_{L^2(\Omega)}
		+\Delta t\sum_{n=0}^{l}
		\|d_t\Lambda_{\xi}^{n+1}\|^2_{L^2(\Omega)}
		\nonumber\\
		&+\Delta t\sum_{n=0}^{l}
		\frac{\tilde{C}_2}{2\beta}\left[ 
		\|\varepsilon(\Lambda_{\mathbf{u}}^{n+1})\|
		_{L^2(\Omega)}+\|\Lambda_{\xi}^{n+1}\|_{L^2(\Omega)}\right]
		+\Delta t\sum_{n=0}^{l}
		\frac{1}{2}\|d_t\operatorname{div}\Lambda_{\mathbf{u}}^{n+1}\|^2_{L^2(\Omega)}\nonumber\\
		&+\|\varepsilon(\Lambda_{\mathbf{u}}^{l+1})\|^2_{L^2(\Omega)}	+\|\Lambda_{\xi}^{l+1}\|^2_{L^2(\Omega)}
		+\|\Lambda_p^{l+\theta}\|^2_{L^2(\Omega)}
		+\|\hat{\Lambda}_p^{l+\theta}\|^2_{L^2(\Omega)}\nonumber\\
		&\left.+(\Delta t)^2\frac{2\kappa_1^2}{\kappa_3}\|\eta_t\|^2_{L^2((0,t_{l+1});L^2(\Omega))}
		+\frac{\mu_f(\Delta t)^2}{3K}\|\eta_{tt}\|^2_{L^2((0,t_{l+1});H^{-1}(\Omega))}\right]\nonumber\\
		&\leq C(\Delta t)^2\left(
		\|\eta_t\|^2_{L^2((0,t_{l+1});L^2(\Omega))}
		+\|\eta_{tt}\|^2_{L^2((0,t_{l+1});H^{-1}(\Omega))}\right)\nonumber\\
		&+Ch^4\left[
		\|\xi_t\|^2_{L^2((0,T);H^2(\Omega))}
		+\|\eta_t\|^2_{L^2((0,T);H^2(\Omega))}
		+\|\nabla\mathbf{u}\|^2_{L^{2}((0,T);H^2(\Omega))}
		\right.\nonumber\\
		&+\left.\|\xi\|^2_{L^{\infty}((0,T);H^2(\Omega))}
		+\|\eta\|^2_{L^{\infty}((0,T);H^2(\Omega))}
		+\|\nabla\mathbf{u}\|^2_{L^{\infty}((0,T);H^2(\Omega))}
		+\|\operatorname{div}\mathbf{u}_t\|^2_{L^{2}((0,T);H^2(\Omega))}
		\right],
	\end{align}
	provided that $\Delta t\leq\frac{\mu_fK_1\kappa_3}{cK_2^2\kappa_1^2}h^2$
	when $\theta=0$ or $\Delta t>0$ when $\theta=1$. 
	Hence, we deduce that \eqref{4.9} holds. The proof if complete.	
\end{proof}
\begin{theorem}\label{them_4.2}
		Let $ \left\lbrace (\mathbf{u}_{h}^{n}, \xi_{h}^{n}, \eta_{h}^{n})\right\rbrace_{n\geq0}  $ be defined by the MFEM, then there holds
	\begin{align}
		\max_{0\leq n\leq l}\left[\right. \|\varepsilon(\mathbf{u}(t_{n+1}))&-\varepsilon(\mathbf{u}_h^{n+1})\|_{L^2(\Omega)}
		+\left.\|\xi(t_{n+1})-\xi_h^{n+1}\|_{L^2(\Omega)}\right.\nonumber\\
		&\left.+\|\eta(t_{n+\theta})-\eta_h^{n+\theta}\|_{L^2(\Omega)}\right] 
		\leq \hat{C}_1(T)\Delta t+\hat{C}_2(T)h^2,
	\end{align}
	\begin{align}
		\left(\Delta t\sum_{n=0}^{l}\frac{K_1}{\mu_{f}}\|\nabla
		p(t_{n+1})-\nabla p_h^{n+1}
		\|^2_{L^2(\Omega)}\right)^{\frac{1}{2}}
		\leq \hat{C}_1(T)\Delta t+\hat{C}_2(T)h,
	\end{align}
	provided that $\Delta t=O(h^2)$ when $\theta=0$ and $\Delta t>0$ when $\theta=1$. Here
	\begin{align*}
		\hat{C}_1(T)&=C\left(
		\|\eta_t\|^2_{L^2((0,T);L^2(\Omega))}
		+\|\eta_{tt}\|^2_{L^2((0,T);H^{-1}(\Omega))}\right),\nonumber\\
		\hat{C}_2(T)&=C\left(
		\|\xi_t\|^2_{L^2((0,T);H^2(\Omega))}
		+\|\eta_t\|^2_{L^2((0,T);H^2(\Omega))}
		+\|\nabla\mathbf{u}\|^2_{L^{2}((0,T);H^2(\Omega))}\right.\nonumber\\
		&+\left.\|\xi\|^2_{L^{\infty}((0,T);H^2(\Omega))}
		+\|\eta\|^2_{L^{\infty}((0,T);H^2(\Omega))}
		+\|\nabla\mathbf{u}\|^2_{L^{\infty}((0,T);H^2(\Omega))}
		+\|\operatorname{div}\mathbf{u}_t\|^2_{L^{2}((0,T);H^2(\Omega))}
		\right).
	\end{align*}
\end{theorem}
\begin{proof}
	The above estimates follow immediately form an application of the triangle inequality on
	\begin{align*}
		&\mathbf{u}(t_{n+1})-\mathbf{u}_h^{n+1}=\Lambda_\mathbf{u}^{n+1}+\Theta_\mathbf{u}^{n+1},\\
		&\xi(t_{n+1})-\xi_{h}^{n+1}=\Lambda_\xi^{n+1}+\Theta_\xi^{n+1}=\hat{\Lambda}_\xi^{n+1}+\hat{\Theta}_\xi^{n+1},\\
		&\eta(t_{n+1})-\xi_{h}^{n+1}=\Lambda_\eta^{n+1}+\Theta_\eta^{n+1}=\hat{\Lambda}_\eta^{n+1}+\hat{\Theta}_\eta^{n+1},\\
		&p(t_{n+1})-p_{h}^{n+1}=\Lambda_p^{n+1}+\Theta_p^{n+1}=\hat{\Lambda}_p^{n+1}+\hat{\Theta}_p^{n+1},
	\end{align*}
	 and appealing to \eqref{4.2}, \eqref{4.5}, \eqref{4.7} and Lemma \ref{lemma4.1}. The proof is complete.
\end{proof}

\subsection{$L^{2}$ error estimation of displacement $\mathbf {u} $}\label{sec-3.3} In this subsection, we prove the $L^{2}$ error estimation of displacement $\mathbf {u} $ by introducing an auxiliary problem.

\begin{theorem}\label{them_4.3}
	Let $ \left\lbrace (\mathbf{u}_{h}^{n}, \xi_{h}^{n}, \eta_{h}^{n})\right\rbrace_{n\geq0}  $ be defined by the MFEM, then there holds
	\begin{align}\label{4-35}
		\|\mathbf{u}(t_{n})-\mathbf{u}_h^{n}\|_{L^{2}(\Omega)} \leq Ch^{3}.
	\end{align}
\end{theorem}
\begin{proof}
	First, 
	we construct a bilinear form
	\begin{align}\label{4.35}
		B((\mathbf{u}, \xi) ;(\mathbf{v},\varphi))
		=(\varepsilon(\mathbf{u}),\varepsilon(\mathbf{v}))
		-(\xi, \operatorname{div} \mathbf{v})
		+(\varphi, \operatorname{div} \mathbf{u})
		+k_{3}(\xi, \varphi), 
	\end{align}
	and it is easy to know that $B$ satisfies the following properties: 
	\begin{align}
		&B((\mathbf{u}, \xi) ;(\mathbf{u}, \xi))=|\mathbf{u}|_{H^{1}( \Omega)}+k_{3}\|\xi\|_{L^2(\Omega)}, \\
		&B((\mathbf{u}, \xi) ;(\mathbf{v}, \varphi))
		\leqslant C\left(|\mathbf{u}|_{H^1(\Omega)}+\|\xi\|_{L^2(\Omega)}\right)\left(|\mathbf{v}|_{H^1(\Omega)}+\|\varphi\|_{L^2(\Omega)}\right),\label{4.37}
	\end{align}
	for all $\mathbf{u},\mathbf{v}\in\mathbf{H}^1(\Omega)$ and $\xi,\varphi\in L^2(\Omega)$.
	
	Then, we introduce the following auxiliary problems 
	\begin{align}
		-\operatorname{div}\varepsilon(\mathbf{w})-\nabla z&=\mathbf{u}-\mathbf{u}_{h}^{n} ,\label{4.38} \\
		\operatorname{div}\mathbf{w}-\kappa_3z&=0 , \label{4.39}
	\end{align}
	in $\Omega$, with $\varepsilon(\mathbf{w})\mathbf{n}-z\mathbf{n}=\mathbf{0}$ on $\partial\Omega$ and $\int_{\Omega}z\mathrm{~dx}=0$.	
	From \cite{B21}, we know that
	\begin{align}\label{2025-6-18-1}
		\|\mathbf{w}\|_{H^2(\Omega)}+\|z\|_{H^1(\Omega)}\leq\|\mathbf{u}-\mathbf{u}_{h}^{n}\|_{L^2(\Omega)}.
	\end{align}	
	Multiply \eqref{4.38} by $\mathbf{u}-\mathbf{u}_{h}^{n}$ and \eqref{4.39} by $\xi-\xi_h^n$, we get
	\begin{align}\label{4.40}
		\|\mathbf{u}-\mathbf{u}_{h}^{n}\|_{L^2(\Omega)}^{2}
		&=(\mathbf{u}-\mathbf{u}_{h}^{n},\mathbf{u}-\mathbf{u}_{h}^{n})\nonumber\\
		&=(-\operatorname{div}\varepsilon(\mathbf{w})-\nabla z,\mathbf{u}-\mathbf{u}_{h}^{n})\nonumber\\
		&=(-\operatorname{div}\varepsilon(\mathbf{w}),\mathbf{u}-\mathbf{u}_{h}^{n})
		-(\nabla z,\mathbf{u}-\mathbf{u}_{h}^{n})\nonumber\\
		&=(\varepsilon(\mathbf{w}),\varepsilon(\mathbf{u}-\mathbf{u}_{h}^{n}))
		+(z,\operatorname{div}(\mathbf{u}-\mathbf{u}_{h}^{n}))\nonumber\\
		&=(\varepsilon(\mathbf{w}),\varepsilon(\mathbf{u}-\mathbf{u}_{h}^{n}))
		+(z,\operatorname{div}(\mathbf{u}-\mathbf{u}_{h}^{n}))\nonumber\\
		&\quad-(\xi-\xi_h^n,\operatorname{div}\mathbf{w})+\kappa_3(\xi-\xi_h^n,z)\nonumber\\
		&=B((\mathbf{u}-\mathbf{u}_h^n, \xi-\xi_h^n) ;(\mathbf{w},z))\nonumber\\
		&=B((\mathbf{u}-\mathbf{u}_h^n, \xi-\xi_h^n) ;(\mathbf{w}-\mathcal{R}_h\mathbf{w},z-\mathcal{Q}_hz))\nonumber\\
		&\quad+B((\mathbf{u}-\mathbf{u}_h^n, \xi-\xi_h^n);(\mathcal{R}_h\mathbf{w},\mathcal{Q}_hz)).
	\end{align}	
	Using \eqref{4.37}, \eqref{4.2}, \eqref{4.5} and  Theorem \ref{them_4.2},  we get
	\begin{align}\label{4.41}
		B&((\mathbf{u}-\mathbf{u}_{h}^{n}, \xi-\xi_{h}^{n}) ;(\mathbf{w}-\mathcal{R}_{h}\mathbf{w}, z-\mathcal{Q}_hz))\nonumber\\
		&\leq C\left(|\mathbf{u}-\mathbf{u}_{h}^{n}|_{H^1(\Omega)}+\|\xi-\xi_{h}^{n}\|_{L^{2}(\Omega)}\right)\left(|\mathbf{w}-\mathcal{R}_{h}\mathbf{w}|_{H^1(\Omega)}+\|z-\mathcal{Q}_hz\|_{L^{2}(\Omega)}\right)
		\nonumber\\
		&\leq Ch^2\left(Ch\|\mathbf{w}\|_{H^2(\Omega) }+Ch\|z\|_{H^1(\Omega)}\right)\nonumber\\
		&\leq Ch^3\left(\|\mathbf{w}\|_{H^2(\Omega) }+\|z\|_{H^1(\Omega)}\right).
	\end{align}	
		Using \eqref{2.24}-\eqref{2.25} and \eqref{3.14}-\eqref{3.15},  setting $\mathbf{v}_h=\mathcal{R}_h\mathbf{w}$ and $\varphi_h=\mathcal{Q}_hz$, we get
	\begin{align}\label{4.42}
		&B((\mathbf{u}-\mathbf{u}_h^n, \xi-\xi_h^n);(\mathcal{R}_h\mathbf{w},\mathcal{Q}_hz))\nonumber\\
		&=(\varepsilon(\mathbf{u}-\mathbf{u}_h^n),\varepsilon(\mathcal{R}_h\mathbf{w}))
		-(\xi-\xi_h^n, \operatorname{div} \mathcal{R}_h\mathbf{w})
		+(\mathcal{Q}_hz, \operatorname{div}( \mathbf{u}-\mathbf{u}_h^n))
		+k_{3}(\xi-\xi_h^n, \mathcal{Q}_hz)\nonumber\\
		&=(\varepsilon(\mathbf{u}-\mathbf{u}_h^n),\varepsilon(\mathcal{R}_h\mathbf{w}))
		-(\xi-\xi_h^n, \operatorname{div} \mathcal{R}_h\mathbf{w})
		+\kappa_1(\eta-\eta_h^{n-1+\theta},\mathcal{Q}_hz)\nonumber\\
		&=(\varepsilon(\mathbf{u}-\mathbf{u}_h^n),\varepsilon(\mathcal{R}_h\mathbf{w}))+
		(\mathcal{N}\varepsilon(\mathbf{u})-\mathcal{N}\varepsilon(\mathbf{u}_h^n),\varepsilon(\mathcal{R}_h\mathbf{w}))\nonumber\\
		&\quad+\kappa_1(\eta-\eta_h^{n-1+\theta},\mathcal{Q}_hz-z)+\kappa_1(\eta-\eta_h^{n-1+\theta},z).
	\end{align}
	Using \eqref{4.21}, \eqref{4.7}, 
	Theorem \ref{them_4.2} and Cauchy-Schwarz inequality, we get
	\begin{align}\label{4.43}
		(\mathcal{N}(\varepsilon(\mathbf{u}))-\mathcal{N}(\varepsilon(\mathbf{u}_h^n)), \varepsilon(\mathcal{R}_h\mathbf{w}))
		&\leq \|\mathcal{N}(\varepsilon(\mathbf{u}))-\mathcal{N}(\varepsilon(\mathbf{u}_h^n))\|_{L^2(\Omega)}\cdot\|\varepsilon(\mathcal{R}_h\mathbf{w})\|_{L^2(\Omega)}\nonumber\\
		&\leq\|\varepsilon(\mathbf{u}-\mathbf{u}_h^n)\|_{L^2(\Omega)}\cdot\|\varepsilon(\mathcal{R}_h\mathbf{w}-\mathbf{w})+\varepsilon(\mathbf{w})\|_{L^2(\Omega)}\nonumber\\
		&\leq  Ch^2\cdot (Ch\|\mathbf{w}\|_{H^2(\Omega)}+Ch\|\mathbf{w}\|_{H^2(\Omega)})
		\nonumber\\
		&\leq  Ch^3\|\mathbf{w}\|_{H^2(\Omega)}.
	\end{align}	
	Similarly,
	\begin{align}\label{4.45}
		(\varepsilon(\mathbf{u}-\mathbf{u}_{h}^{n}),\varepsilon(\mathcal{R}_h\mathbf{w}))
		\leq\|\varepsilon(\mathbf{u}-\mathbf{u}_h^n)\|_{L^2(\Omega)}\cdot\|\varepsilon(\mathcal{R}_h\mathbf{w})\|_{L^2(\Omega)}\leq  Ch^3\|\mathbf{w}\|_{H^2(\Omega)}.
	\end{align}	
	Using \eqref{4.2}, \eqref{12-12-1} and Theorem \ref{them_4.2}, we get
	\begin{align}\label{4.46}
		&\kappa_1\left(\eta-\eta_{h}^{n-1+\theta},\mathcal{Q}_hz-z\right)+\kappa_1\left(\eta-\eta_{h}^{n-1+\theta},z\right)\nonumber\\
		&\leq \kappa_1\|\eta-\eta_{h}^{n-1+\theta}\|_{L^2(\Omega)}\cdot(\|\mathcal{Q}_hz-z\|_{L^2(\Omega)}+\|z\|_{L^2(\Omega)})\nonumber\\
		&\leq Ch^2(Ch\|z\|_{H^1(\Omega)}+Ch\|\nabla z\|_{L^2(\Omega)})\nonumber\\
		&\leq Ch^3\|z\|_{H^1(\Omega)}.
	\end{align}	
	Substituting \eqref{4.41}-\eqref{4.46} into \eqref{4.40} and using \eqref{2025-6-18-1}, we get
	\begin{align}
			\|\mathbf{u}-\mathbf{u}_{h}^{n}\|_{L^2(\Omega)}^{2}
			\leq Ch^3(\|\mathbf{w}\|_{H^2(\Omega)}+\|z\|_{H^1(\Omega)})
			\leq Ch^3\|\mathbf{u}-\mathbf{u}_{h}^{n}\|_{L^2(\Omega)},
	\end{align}	
    which implies that \eqref{4-35} holds. The proof is complete.
\end{proof}

\section{Numerical tests}\label{2025-9-9-3}
In this section, we show some numerical tests to verify theoretical results. We ignore the effects of gravity because the corresponding term does not affect the conclusion.

{\bf Test 1.} Let $\Omega= (0,1)\times(0,1)$, $T=1$,  $\Gamma_{1}= \{ (0,y);~0\leq y\leq1 \}$, $\Gamma_{2}= \{(1,y);~0\leq y\leq1 \}$, $\Gamma_{3}= \{(x,1);~0\leq x\leq1 \}$, $\Gamma_{4} = \{(x,0);~0\leq x\leq1\}$.\\
We take 
\begin{align*}
	\Phi^{'}(\mathrm{dev}\varepsilon(\mathbf{u}))=\frac{1}{2}(1+\mathrm{dev}\varepsilon(\mathbf{u}))^{-\frac{1}{2}}\quad\text{and}\quad
	\kappa(x)=\frac{1}{\lambda}+\frac{\mu}{2}
\end{align*}
to satisfy \eqref{11-25-1} and \eqref{11-25-2}, respectively.
According to \eqref{11-22-1} and \eqref{11-26-2}, we get
\begin{align*}
	&\sigma(\mathbf{u})=\mu((1+\mathrm{dev}\varepsilon(\mathbf{u}))^{-\frac{1}{2}})\varepsilon(\mathbf{u})+
	(\frac{1}{\lambda}+\frac{\mu}{2}-\frac{\mu}{2}(1+\mathrm{dev}\varepsilon(\mathbf{u}))^{-\frac{1}{2}})\operatorname{div}\mathbf{u}\mathbf{I},\\
	&\mathcal{N}(\varepsilon(\mathbf{u}))=\sigma(\mathbf{u})-\frac{1}{\lambda}\operatorname{div}\mathbf{u}\mathbf{I},
\end{align*}
where
\begin{align*}
	\mathrm{dev}\varepsilon(\mathbf{u})=\mathrm{tr}(\varepsilon^{2}(\mathbf{u}))-\frac{1}{2}\mathrm{tr}^{2}(\varepsilon(\mathbf{u})).
\end{align*}
The body force $\mathbf{f}$ and source term $\phi$ as follows
\begin{align*}
	\mathbf{f}=
	\begin{pmatrix}
		f_1\\
		f_2
	\end{pmatrix},\quad
	\phi = 
	(-\frac{c_0}{\pi}+2\alpha\pi t-2\pi t\frac{K}{\mu_{f}})\sin(\pi x + \pi y),
\end{align*}
where
\begin{align*}
	f_1&= \frac{\mu}{4}\pi^4t^6(1+\mathrm{dev}\varepsilon(\mathbf{u}))^{-\frac{3}{2}}\sin(2\pi x)\sin(\pi y-\pi x)(\cos^2(\pi y)-\sin^2(\pi y))\\
	&+\frac{\mu}{4}\pi^4t^6(1+\mathrm{dev}\varepsilon(\mathbf{u}))^{-\frac{3}{2}}\sin(2\pi y)\sin(\pi y+\pi x)(\cos^2(\pi x)-\sin^2(\pi x))\\
	&+\mu \pi ^2t^2(1+\mathrm{dev}\varepsilon(\mathbf{u}))^{-\frac{1}{2}}\cos(\pi x)\cos(\pi y)
	-(\frac{1}{\lambda}+\frac{\mu}{2})\pi^2t^2\cos(\pi x+\pi y)\\
	&-\pi t\alpha\cos(\pi x +\pi y),
\end{align*}
\begin{align*}
	f_2&= \frac{\mu}{4}\pi^4t^6(1+\mathrm{dev}\varepsilon(\mathbf{u}))^{-\frac{3}{2}}\sin(2\pi x)\sin(\pi x+\pi y)(\cos^2(\pi y)-\sin^2(\pi y))\\
	&+\frac{\mu}{4}\pi^4t^6(1+\mathrm{dev}\varepsilon(\mathbf{u}))^{-\frac{3}{2}}\sin(2\pi y)\sin(\pi x-\pi y)(\cos^2(\pi x)-\sin^2(\pi x))\\
	&-\mu \pi ^2t^2(1+\mathrm{dev}\varepsilon(\mathbf{u}))^{-\frac{1}{2}}\sin(\pi x)\sin(\pi y)
	-(\frac{1}{\lambda}+\frac{\mu}{2})\pi^2t^2\cos(\pi x+\pi y)\\
	&-\pi t\alpha\cos(\pi x +\pi y),
\end{align*}
and 
\begin{align*}
	\mathrm{dev}\varepsilon(\mathbf{u})
	=\pi^2t^4(\sin^2(\pi x)\cos^2(\pi y)+\cos^2(\pi x)\sin^2(\pi y)).
\end{align*}
The initial and boundary conditions are
\begin{align*}
	p(x,y,t) = -\frac{t}{\pi} \sin(\pi x+\pi  y) &\qquad\mbox{on }\partial\Omega_T,\\
	u_1(x,y,t)= t^{2}\sin(\pi x)\sin(\pi y) &\qquad\mbox{on }\Gamma_j\times (0,T),\, j=1,3,\\
	u_2(x,y,t) = t^{2}\sin(\pi x)\sin(\pi y) &\qquad\mbox{on }\Gamma_j\times (0,T),\, j=,2,4,\\
	\sigma(\mathbf{u})\bf{n}-\alpha p\bf{n} = \mathbf{f}_1 &\qquad \mbox{on } \p\Ome_T\backslash\Gamma_j,\\
	\mathbf{u}(x,y,0) = \mathbf{0},  \quad p(x,y,0) =0 &\qquad\mbox{in } \Omega.
\end{align*}
It is easy to check that the exact solution is
\begin{align*}
	\mathbf{u}(x,y,t)=
	\begin{pmatrix}
		u_1(x,y,t)\\
		u_2(x,y,t)
	\end{pmatrix}
	=
	\begin{pmatrix}
		t^{2}\sin(\pi x)\sin(\pi y)\\
		t^{2}\sin(\pi x)\sin(\pi y)
	\end{pmatrix},\qquad
	p(x,y,t) =-\frac{t}{\pi} \sin(\pi x + \pi y).
\end{align*}

The term $\mathbf{f}_1$ can be obtained from the exact solution. For simplicity, we employ the Picard iteration method to handle the nonlinear term  $\mathcal{N}(\varepsilon(\mathbf{u}))$.
During the iteration process, the nonlinear part $(1+\mathrm{dev}\varepsilon(\mathbf{u}))^{-\frac{1}{2}}$is "frozen" at its value from the previous iteration step (i.e., treated as a known term), thereby transforming the nonlinear problem into a linear one.
Alternatively, other methods such as Newton's iteration method could also be adopted.


The values of physical parameters are given in Table \ref{table_test1_1}. The other physical variables $\lambda$, $\mu$, $\kappa_1$, $\kappa_2$, $\kappa_3$  can be calculated by \eqref{3.1} and \eqref{3.5}.
\begin{table}[H]
	\begin{center}
		\caption{Values of physical parameters}\label{table_test1_1}
		\begin{tabular}{ccc}
			\hline
			Parameters 	& Description		& Values \\
			\hline
			$E$ 			& Young's modulus 							& 1e6\\
			$\nu$ 		& Poisson ratio 							& 0.499\\
			$\alpha$ 	& Biot-Willis constant 						& 1\\
			$c_0$  		& Constrained specific storage coefficient 	& 1e3\\
			$K$ 			& Permeability tensor 		& 1e-5\\
			$\mu_f$ 	& solvent viscosity 						& 1\\
			\hline
		\end{tabular}
	\end{center}
\end{table}

\begin{table}[H]
	\begin{center}
		\caption{Spatial errors and convergence orders of $\mathbf{u}$ and $p$ with $\Delta t=h^2$ }\label{table_test1_2}
		\begin{tabular}{l c c c c c c c c}
			\hline
			$h$& $\|\mathbf{u}-\mathbf{u}_h\|_{L^2}$  &  order& $\|\mathbf{u}-\mathbf{u}_h\|_{H^1}$ & order& $\|p-p_h\|_{L^2}$ &order& $\|p-p_h\|_{H^1}$&order\\ \hline
			$1/4$ & 6.8683e-02& & 0.7023 & & 2.3439e-2& & 0.4633 &\\
			$1/8$ &1.0679e-02 &2.6851 & 0.2594&1.4365 &4.7212e-3 &2.3116 & 0.2295&1.0131\\
			$1/16$ &1.4777e-03&2.8334&  7.7034e-2&1.7520 &1.1103e-3&2.0969&  0.1137&1.0128\\
			$1/32$ &2.0763e-04&2.8513 & 2.0912e-2&1.8811&2.7114e-4&2.0251 &  5.6734e-2& 1.0037\\ \hline
		\end{tabular}
	\end{center}
\end{table}
\begin{table}[H]
	\begin{center}
		\caption{Time errors and convergence orders of $\mathbf{u}$ and $p$ with $h=\frac{1}{4}$ }\label{test1_table_3}
		\begin{tabular}{l c c c c }
			\hline
			$\Delta t$&  $\|\mathbf{u}-\mathbf{u}_h\|_{H^1}$ & $\rm{order_T}$ & $\|p-p_h\|_{H^1}$&$\rm{order_T}$\\ \hline
			$1/10$ & 0.25517579  & & 0.22948171&\\
			$1/20$ &0.26099611  &	1.8582	&	0.22958984	&1.9838	\\
			$1/40$ & 0.2641283  &	1.9493	&	0.22964434	&	1.9920	\\
			$1/80$ &0.2657351 &	1.9811	&	0.22967171	&1.9960	\\ 
			$1/160$ &0.26654616  &1.9922	&	0.22968541	&	1.9980	\\
			\hline
		\end{tabular}
	\end{center}
\end{table}

\begin{figure}[H]
	\centering
	\includegraphics[width=0.45\textwidth]{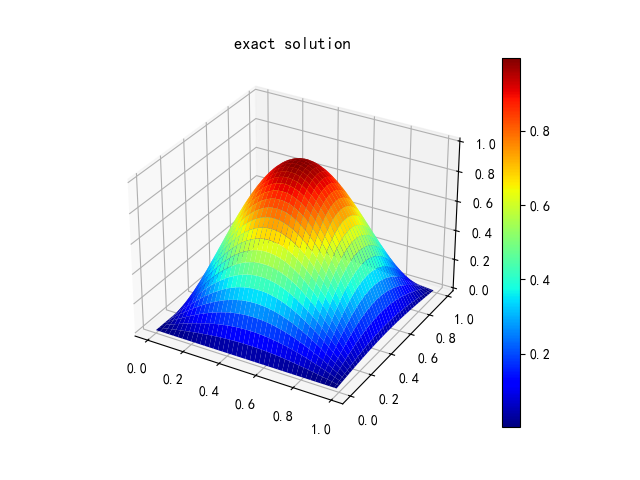}
	\includegraphics[width=0.45\textwidth]{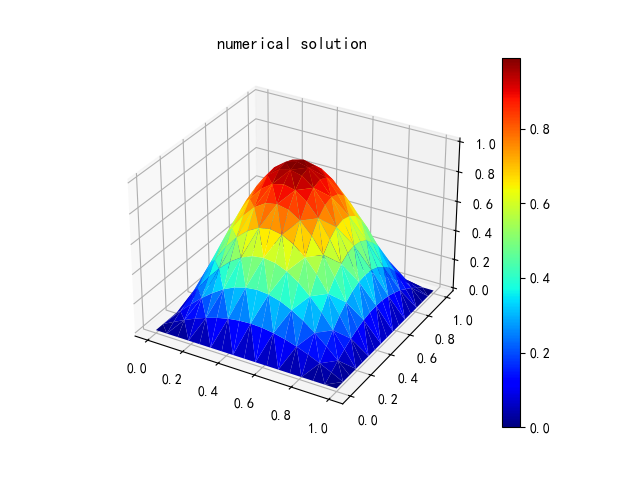}
	\caption{Surface plot of the computed displacement $u_1$ at the terminal time $T$
		.}\label{test1_fig_u1}
\end{figure}
\begin{figure}[H]
	\centering
	\includegraphics[width=0.45\textwidth]{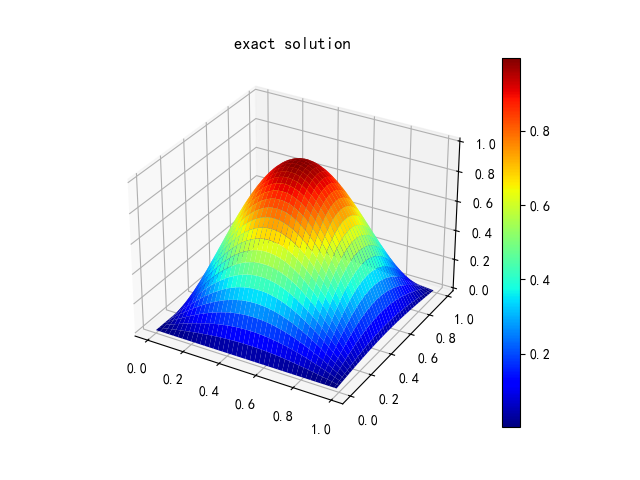}
	\includegraphics[width=0.45\textwidth]{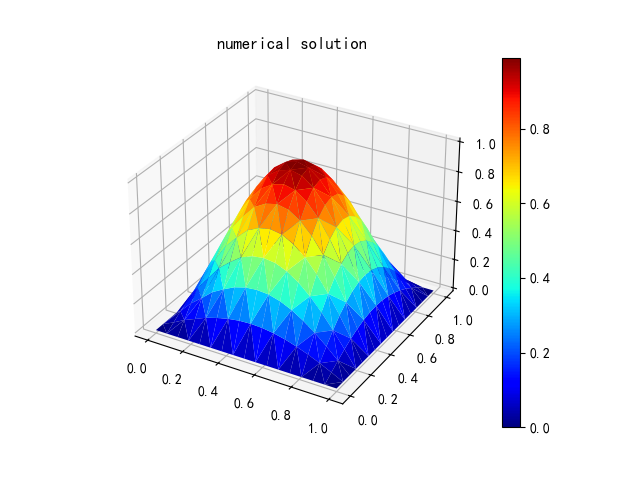}
	\caption{Surface plot of the computed displacement $u_2$ at the terminal time $T$.}\label{test1_fig_u2}
\end{figure}
\begin{figure}[H]
	\centering
	\includegraphics[width=0.45\textwidth]{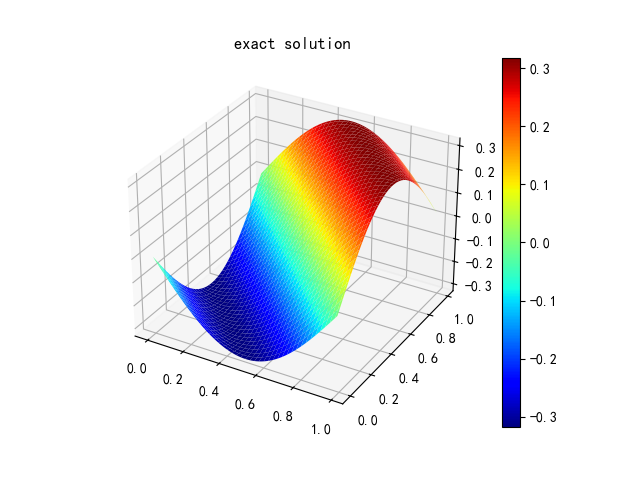}
	\includegraphics[width=0.45\textwidth]{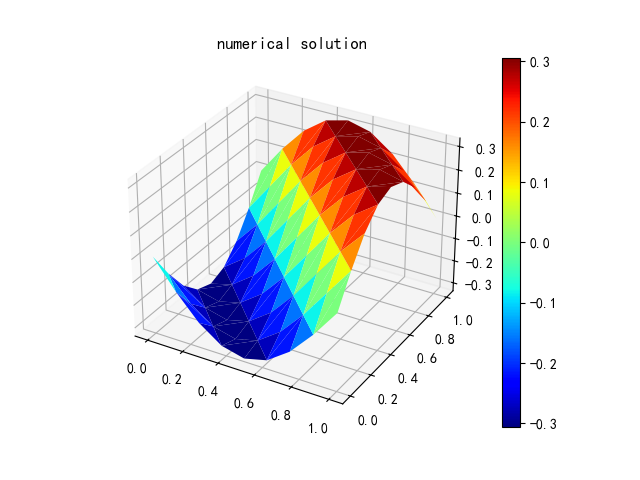}
	\caption{Surface plots of the computed pressure $p$ and exact pressure $p$ at the terminal time $T$.}
	\label{test1_fig_p}
\end{figure}

In this paper, the order of spatial error convergence and time error convergence are defined as 
\begin{align*}
	{\rm{order}}=\frac{\log(R(h)/R(\frac12h))}{\log2},
	\qquad
	{\rm{order_T}}=\abs{\frac{R_h(\Delta t)-R_h(\frac12\Delta t)}{R_h(\frac12\Delta t)-R_h(\frac14\Delta t)}},
\end{align*}
where $R(h)$ is the spatial error with mesh $h$ and $\Delta t=h^2$, and $R_h(\Delta t)$ is the spatial error with respect to time $\Delta t$ for a fixed space mesh $h$.
According to reference \cite{M.Mo2010}, we know $\rm{order_T} \approx 2$ when the corresponding order of convergence in time is $O(\Delta t)$.

Table \ref{table_test1_2} display the computed $ L^2( \Omega)$-norm and $ H^1( \Omega) $-norm errors of $\mathbf{u}$, $p$ and the convergence orders with respect to $h$ at the terminal time $T$.
Evidently, the convergence orders are consistent with Theorem \ref{them_4.2} and Theorem \ref{them_4.3}.

Figure \ref{test1_fig_u1}-Figure \ref{test1_fig_p} show respectively the surface plot of the  computed displacement $u_1$, $u_2$ and pressure $p$ at the terminal time $T$ with mesh parameters $h= \frac{1}{8}$ and $\Delta t= h^2$. They coincide with the exact solution.

{\bf Test 2.} Let $\Omega= (0,1)\times(0,1)$, $T=1$,  $\Gamma_{1}= \{ (0,y);~0\leq y\leq1 \}$, $\Gamma_{2}= \{(1,y);~0\leq y\leq1 \}$, $\Gamma_{3}= \{(x,1);~0\leq x\leq1 \}$, $\Gamma_{4} = \{(x,0);~0\leq x\leq1\}$.\\
We take 
\begin{align*}
	\Phi^{'}(\mathrm{dev}\varepsilon(\mathbf{u}))=1-\frac{1}{2}\mathrm{e}^{-\mathrm{dev}\varepsilon(\mathbf{u})}\quad\text{and}\quad
	\kappa(x)=\frac{1}{\lambda}+\frac{\mu}{2}\mathrm{e}^{-\mathrm{dev}\varepsilon(\mathbf{u})}
\end{align*}
to satisfy \eqref{11-25-1} and \eqref{11-25-2}, respectively. According to \eqref{11-22-1} and \eqref{11-26-2}, we get
\begin{align*}
	&\sigma(\mathbf{u})=\mu(2-\mathrm{e}^{-\mathrm{dev}\varepsilon(\mathbf{u})})\varepsilon(\mathbf{u})+(\frac{1}{\lambda}-\mu+\mu\mathrm{e}^{-\mathrm{dev}\varepsilon(\mathbf{u})})\operatorname{div}\mathbf{u}\mathbf{I},\\
	&\mathcal{N}(\varepsilon(\mathbf{u}))=\sigma(\mathbf{u})-\frac{1}{\lambda}\operatorname{div}\mathbf{u}\mathbf{I},
\end{align*}
where
\begin{align*}
	\mathrm{dev}\varepsilon(\mathbf{u})=\mathrm{tr}(\varepsilon^{2}(\mathbf{u}))-\frac{1}{2}\mathrm{tr}^{2}(\varepsilon(\mathbf{u})).
\end{align*}
The body force $\mathbf{f}$ and source term $\phi$ as follows
\begin{align*}
	\mathbf{f}=&
	\begin{pmatrix}
		-(8\mu t^3(y^2-xy)e^{-2t^2(x-y)^2}+\frac{2t}{\lambda}+2\mu t)+2t\alpha x
		\\
		-(8\mu t^3(x^2-xy)e^{-2t^2(x-y)^2}+\frac{2t}{\lambda}+2\mu t)+2t\alpha y
	\end{pmatrix},\\
	\phi = &
	c_0(x^2+y^2)+2\alpha (x+y)-\frac{4tK}{\mu_{f}}.
\end{align*}
The initial and boundary conditions are
\begin{align*}
	p(x,y,t) = t(x^2+y^2)  &\qquad\mbox{on }\partial\Omega_T,\\
	u_1(x,y,t)= tx^2 &\qquad\mbox{on }\Gamma_j\times (0,T),\, j=1,3,\\
	u_2(x,y,t) = ty^2 &\qquad\mbox{on }\Gamma_j\times (0,T),\, j=,2,4,\\
	\sigma(\mathbf{u})\bf{n}-\alpha p\bf{n} = \mathbf{f}_1 &\qquad \mbox{on } \p\Ome_T\backslash\Gamma_j,\\
	\mathbf{u}(x,y,0) = \mathbf{0},  \quad p(x,y,0) =0 &\qquad\mbox{in } \Omega.
\end{align*}
It is easy to check that the exact solution is
\begin{align*}
	\mathbf{u}(x,y,t)=
	\begin{pmatrix}
		u_1(x,y,t)\\
		u_2(x,y,t)
	\end{pmatrix}
	=t
	\begin{pmatrix}
		x^2\\
		y^2
	\end{pmatrix},\qquad
	p(x,y,t) =t(x^2+y^2).
\end{align*}

Similarly, the term $\mathbf{f}_1$ can be obtained from the exact solution and we also use the Picard iteration method to deal with the nonlinear terms $\mathcal{N}(\mathbf{u}^{n}_h)$.
The values of physical parameters are given in Table \ref{table_1}. The other physical variables $\lambda$, $\mu$, $\kappa_1$, $\kappa_2$, $\kappa_3$  can be calculated by \eqref{3.1} and \eqref{3.5}.
\begin{table}[H]
	\begin{center}
		\caption{Values of physical parameters}\label{table_1}
		\begin{tabular}{ccc}
			\hline
			Parameters 	& Description		& Values \\
			\hline
			$E$ 			& Young's modulus 							& 1e6\\
			$\nu$ 		& Poisson ratio 							& 0.499\\
			$\alpha$ 	& Biot-Willis constant 						& 1\\
			$c_0$  		& Constrained specific storage coefficient 	& 2\\
			$K$ 			& Permeability tensor 		& 1e-5\\
			$\mu_f$ 	& solvent viscosity 						& 1\\
			\hline
		\end{tabular}
	\end{center}
\end{table}

\begin{table}[H]
	\begin{center}
		\caption{Spatial errors and convergence orders of $\mathbf{u}$ and $p$ with $\Delta t=h^2$ }\label{tab31}
		\begin{tabular}{l c c c c c c c c}
			\hline
			$h$& $\|\mathbf{u}-\mathbf{u}_h\|_{L^2}$  &  order& $\|\mathbf{u}-\mathbf{u}_h\|_{H^1}$ & order& $\|p-p_h\|_{L^2}$ &order& $\|p-p_h\|_{H^1}$&order\\ \hline
			$1/4$ & 5.9017e-03& & 6.2865e-2 & & 2.1407e-2& & 0.4082 &\\
			$1/8$ &9.2251e-04 &2.6774 & 1.7457e-2&1.8484 &4.9944e-3 &2.0997 & 0.2077&0.9743\\
			$1/16$ &1.2346e-04&2.9014&  4.5983e-3&1.9246 &1.1309e-3&2.1427&  0.1028&1.0147\\
			$1/32$ &1.5716e-05&2.9737 & 1.1686e-3&1.9762&2.7086e-4&2.0619 &  5.1287e-2& 1.0035\\ \hline
		\end{tabular}
	\end{center}
\end{table}

\begin{table}[H]
	\begin{center}
		\caption{Time errors and convergence orders of $\mathbf{u}$ and $p$ with $h=\frac{1}{4}$ }\label{tab32}
		\begin{tabular}{l c c c c c c c c}
			\hline
			$\Delta t$& $\|\mathbf{u}-\mathbf{u}_h\|_{L^2}$  &  $\rm{order_T}$ & $\|\mathbf{u}-\mathbf{u}_h\|_{H^1}$ & $\rm{order_T}$ & $\|p-p_h\|_{L^2}$ &$\rm{order_T}$& $\|p-p_h\|_{H^1}$&$\rm{order_T}$\\ \hline
			$1/40$ & 9.2268e-4  & & 1.747152e-2	&			&	5.00165e-3	&			&	0.20777426&\\
			$1/80$ &9.2286e-4  &	1.6253	&	1.747657e-2	&	1.9440	&	5.00427e-3	&	1.9497	&0.20778024&1.8297\\
			$1/160$ & 9.2298e-4  &	1.8362	&	1.747917e-2	&	1.9734	&	5.00562e-3	&	1.9765	&0.20778351&1.9204\\
			$1/320$ &9.2304e-4  &	1.9232	&	1.748049e-2	&	1.9871	&5.00630e-3	&	1.9886	&0.20778522&1.9614\\ 
			$1/640$ &9.2307e-4  &	1.9627	&	1.748115e-2	&	1.9936	&	5.00664e-3	&	1.9945	& 0.20778608&1.9809\\
			\hline
		\end{tabular}
	\end{center}
\end{table}

\begin{figure}[H]
	\centering
	\includegraphics[width=0.45\textwidth]{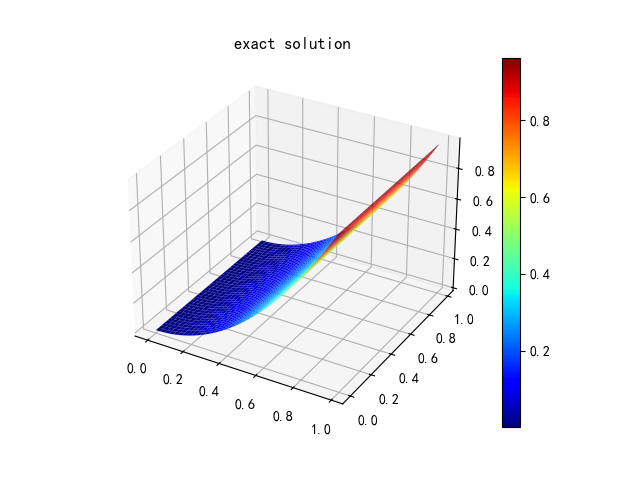}
	\includegraphics[width=0.45\textwidth]{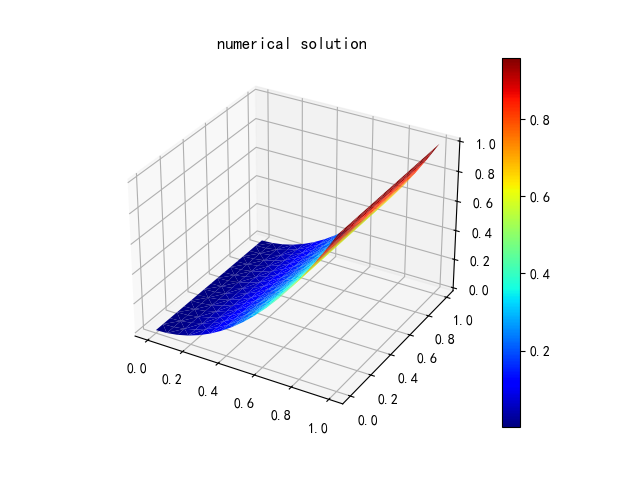}
	\caption{Surface plot of the computed displacement $u_1$ at the terminal time $T$
		.}\label{figure_p31}
\end{figure}
\begin{figure}[H]
	\centering
	\includegraphics[width=0.45\textwidth]{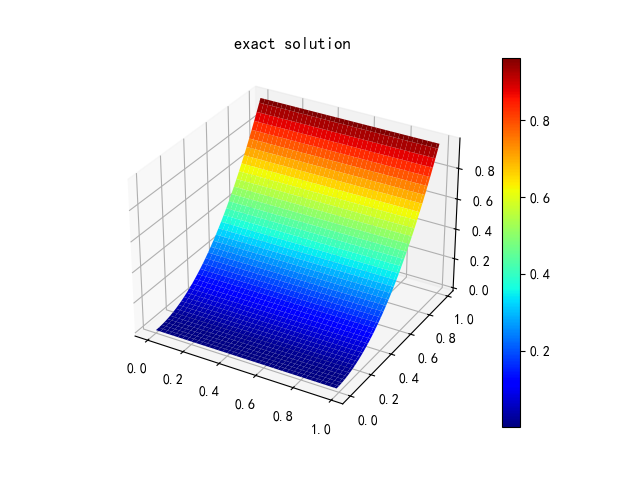}
	\includegraphics[width=0.45\textwidth]{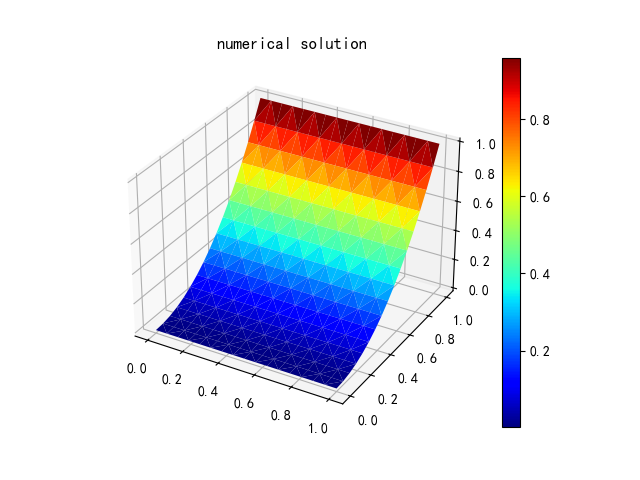}
	\caption{Surface plot of the computed displacement $u_2$ at the terminal time $T$.}\label{figure_p33}
\end{figure}
\begin{figure}[H]
	\centering
	\includegraphics[width=0.45\textwidth]{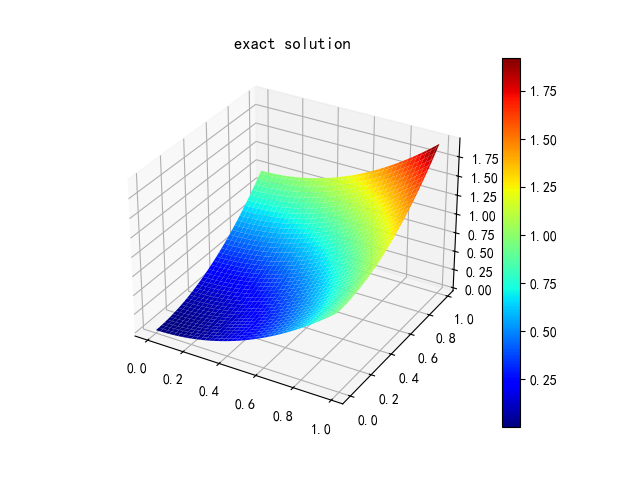}
	\includegraphics[width=0.45\textwidth]{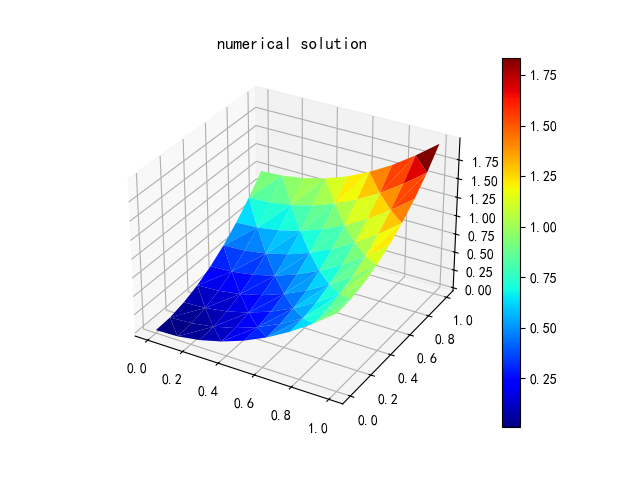}
	\caption{Surface plots of the computed pressure $p$ and exact pressure $p$ at the terminal time $T$.}
	\label{figure_p35}
\end{figure}

Table \ref{tab31} display the computed $L^2( \Omega) $-norm and $H^1( \Omega)$-norm errors of $\mathbf{u}$, $p$ and the convergence orders with respect to $h$ at the terminal time $T$.
Evidently, the convergence orders are consistent with Theorem \ref{them_4.2} and Theorem \ref{them_4.3}.

Figure \ref{figure_p31}-Figure \ref{figure_p35} show respectively the surface plot of the  computed displacement $u_1$, $u_2$ and pressure $p$ at the terminal time $T$ with mesh parameters $h= \frac{1}{8}$ and $\Delta t= h^2$. They coincide with the exact solution.

\section{Conclusion}\label{sec-4}
In this paper,  we propose a multiphysics finite element method for a nonlinear poroelastic model with Hencky-Mises stress tensor. First, we reformulate the original fluid-solid coupling model into a fluid-fluid problem by introducing some new variables, where ($\mathbf{u}$, $\xi$) satisfies a generalized nonlinear Stokes equations and $\eta$ satisfies a diffusion equation. The reformulation model reveals the underlying multiphysics processes and demonstrates that the nonlinearity of $\mathcal{N}(\varepsilon(\mathbf{u}))$ intensifies with increasing $\lambda$. And we establish an energy estimate and prove the existence and uniqueness of the weak solution for the reformulated problem. 
Then, we design a fully discrete multiphysics finite element method to solve the reformulated model with Taylor-Hood element space for the spatial variables $(\mathbf{u}, \xi)$,  linear piecewise polynomials space for $\eta$, and the backward Euler method for time discretization. 
The proposed method overcomes the locking phenomenon and is stable for arbitrary Lagrangian element pairs.
According to the continuity, monotonicity, and coercivity of the nonlinear term $\mathcal{N}(\varepsilon(\mathbf{u}_h^{n}))$, we give the stability analysis and  derive the optimal $L^2$-norm and $H^1$-norm error estimate for displacement vector $\mathbf{u}$ and pressure $p$. In particular, the $L^2$-norm error estimate for the displacement $\mathbf{u}$ is established based on an auxiliary problem and a Poincar$\acute{e}$ inequality, thus filling a gap in the existing literature. Finally, we present numerical tests to verify the theoretical results, where the Picard iteration method is used for the nonlinear terms.


\end{document}